\documentclass[a4paper,12pt]{article}

\usepackage[pdftex]{graphicx}
\usepackage[section] {placeins}
\usepackage{enumerate}
\usepackage[normalem]{ulem}
\usepackage{amsmath}
\usepackage{amsthm}
\usepackage{latexsym}
\usepackage{amsfonts}
\usepackage{amssymb}

\usepackage{tikz-cd}
\usetikzlibrary{graphs}
\usetikzlibrary{arrows.meta}
\usetikzlibrary{shapes.geometric}

\usepackage{authblk}

\usepackage{latexsym,amssymb,amsfonts, amsthm, amsmath}
\usepackage[english]{babel}
\usepackage{bm}
\usepackage{mathrsfs}

\newtheorem{thm}{Theorem}[section]

\newtheorem{lem}[thm]{Lemma}
\newtheorem{cor}[thm]{Corollary}

\newtheorem{conj}[thm]{Conjecture}

\theoremstyle{definition} 
 
\newtheorem{meth}{Method} 
\newtheorem{exa}[thm]{Example}

\newcommand{\Z}{\mathbb{Z}}

\def\BHR{\mathop{\rm BHR}}

\usepackage{fullpage}

\begin{document}

\title{The Buratti-Horak-Rosa Conjecture Holds \\ for Some Underlying Sets of Size Three }

\author{Pranit Chand} 
\author{M.~A.~Ollis\footnote{Corresponding author.  Email: \texttt{matt\_ollis@emerson.edu}}}

\affil{Marlboro Institute for Liberal Arts and Interdisciplinary Studies, Emerson College, Boston, MA 02116, USA} 

\date{}

\maketitle

\begin{abstract}
The Buratti-Horak-Rosa Conjecture concerns the possible multisets of edge-labels of a Hamiltonian path in the complete graph with vertex labels $0, 1, \ldots, {v-1}$ under a particular induced edge-labeling.  The conjecture has been shown to hold when the underlying set of the multiset has size at most~2, is a subset of~$\{1,2,3,4\}$ or~$\{1,2,3,5\}$, or is $\{1,2,6\}$, $\{1,2,8\}$ or  $\{1,4,5\}$, as well as partial results for many other underlying sets.  We use the method of growable realizations to show that the conjecture holds for each underlying set~$U = \{ x,y,z \}$ when $\max(U) \leq 7$ or when $xyz \leq 24$, with the possible exception of $U = \{1,2,11\}$.
We also show that for any even~$x$ the validity of the conjecture for the underlying set $\{ 1,2,x \}$ follows from the validity of the conjecture for finitely many multisets with this underlying set.

\end{abstract}

\section{Introduction}\label{sec:intro}

The setting for the Buratti-Horak-Rosa (BHR) Conjecture is the complete graph~$K_v$ on vertex set~$\{ 0, 1, \ldots, v-1 \}$.  In this graph, define an induced edge-labeling by 
$$\ell(x,y) = \min( |x-y| , v - |x-y|  ),$$ 
called the {\em length} of the edge between~$x$ and~$y$.   If the vertices are arranged in a circle with the natural order, then the length of the edge between~$x$ and~$y$ corresponds to the smallest number of steps around the circle from~$x$ to~$y$.

The edges of a subgraph~$\Gamma$ of $K_v$ give a multiset of lengths in the range $1 \leq \ell(x,y) \leq \lfloor v/2 \rfloor$.    If a multiset~$L$ of edges arises in this way from a Hamiltonian path~$\Gamma$, then $\Gamma$ {\em realizes}~$L$ or $\Gamma$ is a {\em realization} of~$L$ .  We are interested in determining which multisets~$L$ are realizable.

\begin{exa}\label{ex:init}
Let $L = \{ 2, 4^8, 5^3 \}$ (where exponents indicate multiplicity in a multiset).  The Hamiltonian path
$$[ 6, 10, 5, 1, 9, 11, 2, 7, 3, 12, 8, 4, 0 ]$$
in $K_{13}$
has edge-lengths
$$ 4, 5,4,5,2,4,5,4,4,4,4,4 $$
and so realizes~$L$.
\end{exa}

The BHR Conjecture---originally posed by Buratti for prime~$v$~\cite{West07}, generalized to arbitrary~$v$ by Horak and Rosa~\cite{HR09}, and given an equivalent more succinct formulation by Pasotti and Pellegrini~\cite{PP14b}---is as follows:

\begin{conj}[BHR Conjecture]\label{conj:bhr}
For any multiset~$L$ of size $v-1$ with underlying set~$U \subseteq \{ 1, \ldots, \lfloor v/2 \rfloor \}$, there is a realization of~$L$ in $K_v$ if and only if for any divisor~$d$ of~$v$ the number of multiples of~$d$ in~$L$ is at most~$v-d$.
\end{conj}

The condition on divisors in the conjecture is necessary~\cite{HR09}.  Call a multiset~$L$ whose largest element is at most~$\lfloor v/2 \rfloor$ and that meets this necessary condition {\em admissible}.  Denote the BHR Conjecture for~$L$ by~$\BHR(L)$.

The BHR Conjecture has been completely solved for very few underlying sets.  Two early papers independently covered those of size at most~2~\cite{DJ09, HR09}, followed shortly afterwards by a proof for~$\{1,2,3\}$, the first three element set to be solved~\cite{CD10}.  Subsequent work covers subsets of~$\{1,2,3,4\}$~\cite{OPPS2} or~$\{1,2,3,5\}$~\cite{PP14b}, $\{1,2,6\}$ and $\{1,2,8\}$~\cite{PP14}, and $\{1,4,5\}$~\cite{OPPS2}.  In addition to this, there are strong partial results for a wide variety of underlying sets.  See~\cite[Theorem~1.2]{OPPS2} for a summary of the position as of mid-2021.  We shall use one of these partial results in Sections~\ref{sec:12x} and~\ref{sec:alg}: when~$x$ is even, the multiset $\{ 1^a, 2^b, x^c \}$ is realizable if $a+b \geq x-1$~\cite{PP14}.

The BHR Conjecture fits into a network of related problems, often related to graph decompositions of various types, for example the {\em Seating Couples Around the King's Table} problem~\cite{PM09} and determining whether particular Cayley graphs of the cyclic group admit cyclic decompostions~\cite{PP14}.    Seamone and Stevens~\cite{SS12} ask a more general question that encompasses arbitrary groups and spanning trees, for which the specialization to the cyclic group and Hamiltonian paths is equivalent to the BHR Conjecture.  Perhaps the oldest instance of a problem of this type can be connected to the Walecki Construction of~1892, see~\cite{Alspach08}, which may be viewed as a verification of the BHR Conjecture when the multiset~$L$ is $\{ 1^2, 2^2, \ldots, (m-1)^2, m \}$ or $\{ 1^2, 2^2, \ldots, m^2 \}$.  See~\cite{OPPS} for further discussion and references concerning these connections.

In the next section we introduce the notion of a growable realization, as developed in~\cite{OPPS2}, which is our primary tool, and prove some new technical results providing constraints on when they can exist.    For a given underlying set~$U$, the approach divides the work into $\pi(U) = \prod_{x\in U} x$ cases.  

In Section~\ref{sec:max5} we describe methods that each take a finite list of growable realizations for a multiset with underlying set~$U$ to completely prove the BHR~Conjecture for one of the $\pi(U)$ cases.  We use these methods to complete the proof of the BHR Conjecture when the underlying set has size~3 and the largest element is at most~5. 

In Section~\ref{sec:12x} we show that realizations with some parameters must be growable.  This lets us build on the work of Pasotti and Pellegrini in~\cite{PP14} for the problem with~$U = \{ 1,2,x \}$ for some even $x$.  For a fixed~$x$ in this instance, we show that if the BHR Conjecture fails, then it must fail for a multiset~$L = \{ 1^a, 2^b, x^c \}$ of size~$v-1$ with $a+b < x-1$ and $v < x^2 -1$.  This is used to prove the conjecture for underlying sets~$\{1,2,10\}$ and~$\{1,2,12\}$.

In Section~\ref{sec:alg} we develop and implement an algorithm for generating a finite set of realizations to prove one of the $\pi(U)$ cases.   Using this we are able to prove the conjecture for a variety of three-element underlying sets.  

Taking the results of these  approaches together, we prove Theorem~\ref{th:main} in Section~\ref{sec:final}. 

\begin{thm}\label{th:main}
Let~$U$ be a set of size~$3$ with $\max{(U)} \leq 7$ or $\pi(U) \leq 24$.  With the possible exception of $U = \{1,2,11\}$, the BHR Conjecture holds for multisets with underlying set~$U$. 
\end{thm}

Existing results (stated above) give~10 instances of underlying sets~$U$ of size~3 for which the BHR~Conjecture is known to hold.  Theorem~\ref{th:main} adds a further~27 underlying sets to this list.

\section{Growable Realizations}\label{sec:grow}

If a multiset~$L$ has an ``$x$-growable" realization for~$L$, then we can use it to produce a realization for $L \cup \{ x^{kx} \}$ for any~$k \geq 1$.   When a realization is ``$x$-growable" for multiple values of~$x$, this is a powerful tool for constructing realizations for a wide variety of multisets.  In this section we give the necessary definitions and constructions that we use.  Proofs of their correctness may be found in~\cite{OPPS2}.  

The method works via an embedding of $K_v$ into~$K_{v+x}$, for some $x \leq v/2$.    For some fixed~$m \in \{ 0, 1, \ldots, v-1 \}$, define the {\em growth embedding} by
$$
y \mapsto
\begin{cases}
y \hspace{2mm}  \text{when } y\leq m,  \\
y+x \hspace{2mm}  \text{otherwise.}
\end{cases}
$$

Let $\bm{h} = [h_1, \ldots, h_v]$ be a realization of a multiset~$L$.  Take $x$ and $m$ as in the growth embedding.  If 
each $y$ with $m-x < y \leq m$ is incident with exactly one edge whose length is increased by the growth embedding and there are no other edges whose lengths are increased, then say that $\bm{ h}$ is {\em $x$-growable at $m$}.  Let $X = \{ x_1, x_2, \ldots, x_k \}$.  If $\bm{ h}$ is $x_i$-growable for each $x_i \in X$ then say that $\bm{ h}$ is {\em $X$-growable}.

\begin{thm}\label{th:grow}{\rm \cite{OPPS2}}
Suppose a multiset~$L$ has an $X$-growable realization. Then for each~$x \in X$, the multiset $L \cup \{ x^x \}$ has an  $X$-growable realization.
\end{thm}

\begin{proof}[Proof Construction]

Let $\bm{ g} = [g_1, \ldots, g_v]$ be an $X$-growable realization of a multiset~$L$.  Take $x \in X$ with $\bm{ g}$  $x$-growable at $m$.  

Apply the growth embedding to obtain a (non-Hamiltonian) path $\bm{ h'}$ in $K_{v+x}$.  Each element $y$ with $m-x < y \leq m$ is adjacent to exactly one element~$z$ in the original realization that is mapped to~$z+x$, meaning that $y$ is adjacent to~$z+x$ in~$\bm{ h'}$.   Insert the element~$y+x$ between them  for each of these~$y$ values  to give a Hamiltonian path~$\bm{ h}$ in~$K_{v+x}$.  The Hamiltonian path~$\bm{ h}$ is the required $X$-growable realization of $L \cup \{ x^x \}$.
\end{proof}

Repeated applications of Theorem~\ref{th:grow} gives the following corollary.

\begin{cor}\label{cor:multigrow}{\rm \cite{OPPS2}}
Let $X = \{x_1, x_2, \ldots, x_k\}$.
If a multiset~$L$ has an $X$-growable realization, then the multiset $L \cup \{ x_1^{\ell_1x_1}, x_2^{\ell_2x_2}, \ldots, x_2^{\ell_2x_2} \}$ has an  $X$-growable realization for any $\ell_1, \ell_2, \ldots, \ell_k \geq 0$.
\end{cor}

\begin{exa}\label{ex:grow}
The realization for $\{ 2, 4^8, 5^3\}$ in Example~\ref{ex:init} is $\{4,5\}$-growable.  It is $4$-growable at~3,~4 and~5 and $5$-growable at~$6$.  If we apply Theorem~\ref{th:grow} twice, first for 4-growability (at $3$) and second for 5-growability (at 10, the point to which the growth embedding moves the 5-growability value), we obtain the $\{4,5\}$-growable realization
$$[ 10, 15, 19, 14, 9, 5, 1, 18, 20, 2, 6, 11, 16, 12, 7, 3, 21, 17, 13, 8, 4, 0 ]  $$
for $\{ 2, 4^{12}, 5^8 \}$. 
It is 4-growable at~3 (and other values) and 5-growable at~10 (and other values).
\end{exa}

The following lemma gives some instances of admissible multisets that do not have particular growable realizations.  

\begin{lem}\label{lem:inadgrow}
Let $L = \{ x^a, y^b, z^c \}$ with $x < y < z$.  Set $v = a+b+c +1$ and $X \subseteq \{x,y,z\}$.  There is no $X$-growable realiation for~$L$ in the following situations: 
\begin{enumerate}
\item $z \in X$ and $v = 2z$,
\item $z \in X$ and $c < z-y$,
\item $x \in X$ and $a+b < z-x-1$,
\item $y \in X$ and $a+b < z-y-1$.
\end{enumerate}
\end{lem}

\begin{proof}
Consider Item~1.  Suppose~$v = 2z$ and $\bm{h}$ is a realization that is~$z$-growable at~$m$.  To satisfy the definition, the vertex $m-z+1$ must be connected to a vertex with label at least~$m+1$ and this must give an edge that is lengthened by the growth embedding.  However, as $v=2z$,  such edges are not lengthened by the growth embedding so we cannot have a $z$--growable realization. 

Now consider Item~2.  Suppose $c < z-y$ and we have a realization~$\bm{h}$ that is~$z$-growable at~$m$.  The vertices $m-z+1$ to $m-y$ must each be incident with an edge that is lengthened, the only option for which is one of length~$z$.  However there are~$z-y > c$ vertices in the list, so this is impossible.

For Item~3, suppose $a+b < z-x-1$ and $\bm{h}$ is a realization that is~$x$-growable at~$m$.  Each edge of length $z$ joins two elements~$p$ and $p+z$ and $x$-growability means that we cannot have $m- z < p  < m-x+1$.  Hence $c \leq v - z + x$.  Substituting $v = a+b+c+1$ gives the contradictory $a+b \geq z-x-1$.   The argument for Item~4 is the same, with $y$ in place of~$x$.
\end{proof}

Lemma~\ref{lem:inadgrow} is useful in future sections for avoiding fruitless searches for particular growable realizations.

\section{Underlying Sets with Largest Element~5}\label{sec:max5}

In this section we describe in more detail some possible ways we might move from a small list of growable realizations to a complete solution for one of the $\pi(U)$ cases.  It is often the case that a method misses some small subcases.  The following result, proved by a computation of Meszka~\cite{Meszka}, deals with this situation.

\begin{thm}\label{th:small}
Let $L$ be an admissible multiset of size~$v-1$ with underlying set~$U$. If $v \leq 19$ or $v=23$ then $\BHR(L)$ holds.  
\end{thm}


Throughout this section, fix $U = \{ x,y,z \}$ and $L = \{ x^a,y^b,z^c \}$.    Corollary~\ref{cor:multigrow} divides the problem naturally into $\pi(U) = xyz$ cases.  To facilitate discussion of this, as in~\cite{OPPS2} we write
$$(r_1, r_2, r_3) \equiv (s_1, s_2, s_3) \pmod{(t_1, t_2, t_3)}$$ 
to mean that $r_i \equiv s_i \pmod{t_i}$ for $1 \leq i \leq 3$.  We further write $$(r_1,r_2, r_3) \preceq (s_1,s_2, s_3)$$ to mean that  $r_i \leq s_i$ for   $1 \leq i \leq 3$. 

Let $a_0$ be in the range $0 < a_0 \leq x$ with $a_0 \equiv a \pmod{x}$.  Define~$a_i = a_0 + ix$ for all $i > 0$.  Make similar definitions for $b_0$ and $b_i$ with respect to~$y$ and for $c_0$ and $c_i$ with respect to~$z$.  
 
Depending which small growable realizations exist within a given case, there are various methods that we might deploy to prove the BHR Conjecture for that case.  For each method the goal is the same:  produce a finite set~$H$ of realizations for which given any  $L = \{ x^a,y^b,z^c \}$ in that case there is an $\bm{h} \in H$ that realises $\{ x^{a'},y^{b'},z^{c'} \}$ with the properties that $(a',b',c') \preceq (a,b,c)$ and if $a' \neq a$ (respectively $b \neq b'$ or $c \neq c'$) then $\bm{h}$ is $x$-growable (respectively $y$- or $z$-growable).

We describe the methods in approximate increasing order of number of realizations needed (it can only be approximate as some methods have a variable number of required realizations).   Some earlier methods are special cases of later ones,
but we separate them out as the special cases we list are both simpler and more frequently used.

Note that choosing different orderings of~$x$,~$y$ and~$z$ gives alternative ways to implement many of the methods.  The choice described corresponds to the most frequently used configuration when $x < y < z$.  For brevity and ease of reading, we denote the multiset $L = \{ x^{a_i},y^{b_j},z^{c_k} \}$ by the triple~$(a_i,b_j,c_k)$.

\begin{meth}
Find a $U$-growable realization~${\bm h_1}$ for $(a_0,b_0,c_0)$.  Then~${\bm h_1}$ covers all $(a,b,c)$ in this case with $(a,b,c) \succeq (a_0,b_0,c_0)$, which is all of them.
\end{meth}

\begin{meth}
Find a $U$-growable realization~$\bm{h_1}$ for $(a_0, b_0, c_1)$ and a $\{x,y\}$-growable realization~$\bm{h_2}$ for $(a_0,b_0,c_0)$.  The realization~$\bm{ h_1}$ covers all subcases with $c \neq c_0$ and $\bm{h_2}$ covers those with $c  = c_0$.
\end{meth}

\begin{meth}
Find a $U$-growable realization~$\bm{h_1}$ for $(a_0, b_0, c_1)$ and $\{x,y\}$-growable realizations~$\bm{h_2, h_3}$ for $(a_i,b_0,c_0)$ and $(a_0,b_j,c_0)$ respectively, where $i,j >0$.   The realization~$\bm{ h_1}$ covers all subcases with $c \neq c_0$, $\bm{ h_2}$ covers all subcases with $c=0$ and $a \geq a_i$, and $\bm{ h_3}$ covers all subcases with $c=0$ and $b \geq b_j$.  Hence we have only finitely many exceptions: those with $(a,b,c) \preceq (a_i -x , b_j -y, c_0)$.  The largest of these has $v = a_i+b_j+c_0 -x-y +1$.  In all instances we consider, these exceptions are covered by Theorem~\ref{th:small}.

If $(a_0,b_0,c_k)$ is inadmissible for all $c_k$, then we may replace $\bm{h_1}$ with $\bm{h_4, h_5}$ for $(a_1,b_0,c_1)$ and $(a_0,b_1,c_1)$ respectively.  This covers all subcases with $c \geq c_k$.
\end{meth}









\begin{meth}  Find a $U$-growable realization~$\bm{h_1}$ for $(a_0, b_0, c_k)$ and for each~$\ell$ with $0 \leq \ell < k$ find $\{x,y\}$-growable realizations~$\bm{h_{2,\ell}, h_{3,\ell}}$ for $(a_{i_\ell},b_0,c_\ell)$ and $(a_0,b_{j_\ell},c_\ell)$ respectively, where $i_\ell,j_\ell >0$.   The realization~$\bm{ h_1}$ covers all subcases with $c \geq c_k$.  Then  $\bm{ h_{2,\ell}}$ covers all subcases with $c=c_\ell$ and $a \geq a_{i_\ell}$ and $\bm{ h_{3,\ell}}$ covers all subcases with $c=c_\ell$ and $b \geq b_{j_\ell}$.   

Let $i = \max_\ell(i_\ell)$ and $j = \max_\ell(j_\ell)$.  We have only finitely many exceptions as any exception must have $(a,b,c) \preceq (a_i -x , b_j -y, c_k -z)$.  The largest of these has $v = a_i+b_j+c_k -x-y-z +1$.  In all instances we consider, these exceptions are covered by Theorem~\ref{th:small}.
\end{meth}

\begin{meth}
Find a $U$-growable realization~$\bm{h_1}$ for $(a_1, b_1, c_1)$.  Find realizations~$\bm{h_2, h_3, h_4}$ for $(a_i,b_0,c_0)$
$(a_0, b_j, c_0)$ and $(a_0, b_0, c_k)$ that are $x$-, $y$- and $z$-growable respectively, unless the subcase is always inadmissible.   Find a $\{y,z\}$-growable realization~$\bm{h_5}$ for $(a_0,b_1, c_{1})$,  find an $\{x,z\}$-growable realization~$\bm{h_6}$ for $(a_1,b_0, c_{1})$, and  find an $\{x,y\}$-growable realization~$\bm{h_7}$ for $(a_1,b_{1}, c_{0})$.

If $a = a_0$ and $b = b_0$ then either $(a,b,c) \succeq (a_0,b_0,c_k)$, and so is covered by~$\bm{h_4}$, or has $v \leq a_0 + b_0 + c_k - z + 1$. In all cases we consider,  these potential small exceptions are covered by Theorem~\ref{th:small}.  A similar argument covers the other subcases when two of $a,b,c$ are as small as possible.  Now suppose that $a = a_0$ but $b > b_0$ and $c > c_0$.  Then $(a,b,c) \succeq (a_0,b_1, c_{1})$  and so $\bm{h_5}$  suffices.  A similar argument covers the other subcases when exactly one of $a,b,c$ is as small as possible.
Finally, if none of $a,b,c$ are as small as possible, then $(a,b,c) \succeq (a_1, b_1, c_1)$ and the subcase is covered by~$\bm{h_1}$.  

\end{meth}

While more general methods may certainly be devised and will be necessary for some underlying sets, these methods are sufficient to prove the main result of this section.

\begin{thm}\label{th:245_345}
If $U = \{ 2,4,5 \}$ or $\{ 3,4,5 \}$ then $\BHR(L)$ holds for multisets~$L$ with underlying set~$U$.
\end{thm}

\begin{proof}

For each~$U$ we give tables of realizations which, in combination with the methods described in this section and Theorem~\ref{th:small} imply the result.  Tables~\ref{tab:245a} to~\ref{tab:245d} cover $U = \{ 2,4,5 \}$ and Tables~\ref{tab:345a} to~\ref{tab:345d} cover $U = \{ 3,4,5 \}$.  

Each problem has $\pi(U)$ cases, which is~40 and~60 for $\{2,4,5\}$ and~$\{3,4,5\}$ respectively, and the tables cover these cases in lexographic order of~$(a_0, b_0, c_0)$.    Note that this means that the cases that are more difficult (and hence require more intricate methods) tend to be earlier in the tables.
\end{proof}

\begin{table}[tp]
\caption{Growable  realizations for $\{2^a, 4^b,5^c\}$ separated by congruence modulo~$(2,4,5)$: case $(1,1,1)$ to $(1,2,3)$.
They are $x$-growable at $m_x$ in $\bm{m} = (m_2,m_4,m_5)$.}\label{tab:245a}
\begin{center}
\begin{footnotesize}
$$\begin{array}{llllll}
\hline
\text{Case} & (a,b,c) & \text{Realization}  & \bm{m} & \text{Method} \\ 
\hline
(1,1,1) & (3, 5, 6) & [ 5, 7, 2, 13, 3, 8, 4, 14, 12, 1, 6, 10, 0, 11, 9 ] &  ( 10, 4, 5 ) & 5 \\
& (7, 1, 1) &  [ 1, 9, 7, 3, 5, 0, 2, 4, 6, 8 ] & ( 7,-,-) &   \\  
& ( 1, 1, 11) & [ 5, 0, 10, 1, 6, 11, 2, 7, 9, 4, 13, 8, 3, 12 ] &  (-,-, 4)  & \\  
& (1, 5, 6 ) & [ 1, 6, 2, 10, 5, 0, 9, 11, 7, 12, 3, 8, 4 ] & (-, 3, 4) & \\   
& (3, 1, 6) & [ 9, 0, 5, 3, 10, 4, 6, 1, 7, 2, 8 ] & (8, -, 4 ) & \\  
& (3, 5, 1) &  [ 0, 6, 2, 8, 4, 9, 1, 3, 7, 5 ] & (2, 5,-) & \\  
\hline
(1,1,2) & ( 3, 5, 7 ) &  [ 11, 9, 5, 10, 6, 1, 15, 13, 8, 3, 14, 2, 7, 12, 0, 4 ] & ( 13, 3, 9 ) & 5 \\  
& ( 7, 1, 2 )& [ 10, 8, 6, 4, 2, 7, 9, 0, 5, 1, 3 ] & ( 7,-,-) \\  
& ( 1, 9, 2 ) &  [ 7, 3, 12, 8, 0, 4, 9, 11, 2, 6, 10, 1, 5 ] & ( -, 3, -) \\
& (  1, 5, 7 ) &  [ 6, 11, 1, 10, 0, 5, 7, 2, 12, 8, 3, 13, 4, 9 ] & (-,5, 6 ) \\
& (  3,1, 7 ) &  [ 11, 4, 6, 1, 9, 2, 7, 0, 5, 3, 10, 8 ] & ( 8,-, 4) \\  
& (  3,  5, 2 ) &  [ 10, 1, 6, 0, 7, 3, 5, 9, 2, 4, 8 ] & (2, 6,-) \\   
\hline
(1,1,3) & (3, 5, 8 ) &  [ 11, 15, 10, 6, 4, 16, 3, 8, 13, 1, 14, 9, 5, 7, 12, 0, 2 ] & ( 4, 8, 11) & 5 \\  
&(  5,1, 3 )& [ 1, 9, 4, 0, 2, 7, 5, 3, 8, 6 ] & ( 6,-,-) & \\ 
&(  1, 5, 3 )& [ 8, 3, 7, 2, 6, 0, 4, 9, 1, 5 ] & (-, 3,-) \\  
&( 1, 1, 8 )& [ 2, 7, 1, 6, 0, 5, 9, 4, 10, 8, 3 ] & (-,-,4) \\  
&( 1,  5, 8 )& [ 6, 11, 0, 10, 5, 1, 3, 13, 2, 7, 12, 8, 4, 14, 9 ] & (-, 7, 9) \\   
&( 3, 1, 8 )& [ 12, 4, 6, 1, 10, 2, 7, 5, 0, 8, 3, 11, 9 ] & ( 9,-, 4) \\ 
&( 3, 5, 3 )&  [ 3, 7, 11, 1, 6, 2, 0, 8, 10, 5, 9, 4 ] & (1, 4,-) \\ 
\hline
(1,1,4) & ( 5, 1, 4 ) & [ 4, 6, 1, 10, 3, 5, 0, 9, 7, 2, 8 ] & (8, 3, 4) & 4 \\   
&( 1, 1, 9 )& [ 6, 11, 1, 8, 3, 10, 5, 0, 4, 9, 2, 7 ] & (-, 5, 6) \\
&( 1, 5, 4 )&  [ 1, 6, 0, 5, 9, 2, 7, 3, 10, 8, 4 ] & ( -, 3, 4 ) \\   
&(  3, 1, 9 )& [ 7, 12, 3, 8, 13, 1, 10, 5, 0, 2, 6, 11, 9, 4 ] & ( -, 5, 7 ) \\ 
&( 3, 5, 4 )&  [ 6, 11, 2, 7, 3, 1, 5, 9, 4, 0, 8, 10, 12 ] & ( -, 3, 6) \\   
\hline
(1,1,5) & (5, 1, 5) &  [ 9, 2, 7, 5, 0, 10, 8, 3, 11, 1, 6, 4 ] & ( 9, 3, 4 ) & 4 \\  
&( 1, 1, 10)& [ 4, 9, 1, 6, 11, 3, 5, 0, 8, 12, 7, 2, 10 ] & (-3, 4) \\  
&(1, 5, 5 )& [ 1, 5, 0, 8, 10, 6, 11, 3, 7, 2, 9, 4 ], [ 4, 5 ] & (-,3, 4) \\ 
&(3, 1, 10 )& [ 0, 5, 10, 12, 7, 2, 13, 8, 3, 1, 11, 6, 4, 14, 9 ] & (-, 5, 9) \\  
&( 3, 5, 5 )&  [ 6, 8, 3, 1, 11, 7, 2, 12, 10, 5, 0, 4, 9, 13 ] & (-, 5, 6) \\  
\hline
(1,2,1) &  ( 3, 6, 6 ) & [ 5, 0, 12, 14, 9, 4, 2, 6, 1, 13, 8, 10, 15, 11, 7, 3 ] & (1, 10, 4) & 5 \\  
&( 7,  2, 1 )& [ 1, 10, 8, 6, 2, 4, 0, 9, 3, 5, 7 ] &  (6,-,-) \\  
&( 1, 10, 1 )& [ 0, 4, 9, 5, 1, 10, 8, 12, 3, 7, 11, 2, 6 ] & (-, 3, -) \\   
&( 1, 6, 6 )& [ 10, 6, 1, 11, 7, 3, 12, 2, 0, 5, 9, 4, 13, 8 ] & (-, 6, 8) \\  
&( 3, 2, 6 )& [ 10, 5, 0, 8, 1, 3, 7, 2, 4, 9, 11, 6 ] & ( 1,-, 6) \\  
&( 3, 6, 1 )& [ 6, 10, 3, 7, 9, 0, 4, 2, 8, 1, 5 ] & ( 7, 3, -)  \\
\hline
(1,2,2) & ( 3, 6, 7 ) &  [ 9, 14, 12, 0, 5, 1, 13, 11, 16, 3, 7, 2, 15, 10, 6, 8, 4 ] &( 10, 3, 4) & 5 \\
&( 5,  2, 2)&  [ 9, 4, 6, 8, 0, 2, 7, 1, 5, 3 ] & (1,-,-) \\  
&( 1, 6, 2 )&  [ 2, 6, 0, 5, 9, 1, 7, 3, 8, 4 ] & (-,5,-) \\  
&( 1, 2, 7 )& [ 4, 9, 3, 7, 1, 6, 2, 8, 10, 5, 0 ] &  (-,-, 4) \\   
&( 1, 6, 7 )& [ 5, 0, 11, 7, 2, 12, 8, 3, 1, 6, 10, 14, 4, 9, 13 ] & (-, 9, 4) \\ 
&( 3,  2, 7 )& [ 1, 6, 11, 0, 9, 4, 2, 10, 5, 7, 12, 8, 3 ]&  (2,-, 7) \\  
&( 3,  6, 2 )& [ 6, 11, 3, 7, 9, 1, 5, 0, 8, 10, 2, 4 ] & (7, 3,-) \\ 
\hline
(1,2,3) & (  5, 2, 3 ) & [ 7, 9, 3, 5, 0, 2, 6, 1, 10, 8, 4 ] & ( 8, 3, 4) & 4 \\  
&( 1, 2, 8 )& [ 3, 7, 2, 9, 4, 11, 6, 1, 8, 0, 10, 5 ] & (-,4, 5) \\ 
&( 1, 6, 3 )& [ 4, 8, 1, 6, 10, 3, 7, 2, 9, 0, 5 ] & (-, 3, 4 ) \\
&( 3, 2, 8 )& [ 6, 11, 13, 4, 8, 3, 1, 10, 5, 0, 9, 7, 2, 12 ] & (-, 5, 6 ) \\ 
&( 3, 6, 3 )& [ 7, 11, 0, 8, 3, 12, 1, 5, 9, 4, 2, 6, 10 ] & (-, 4, 7 ) \\ 
\hline
\end{array}$$
\end{footnotesize}
\end{center}
\end{table}

\begin{table}[tp]
\caption{Growable  realizations for $\{2^a, 4^b,5^c\}$ separated by congruence modulo~$(2,4,5)$: case $(1,2,4)$ to $(1,4,4)$.
They are $x$-growable at $m_x$ in $\bm{m} = (m_2,m_4,m_5)$.}\label{tab:245b}
\begin{center}
\begin{footnotesize}
$$\begin{array}{llllll}
\hline
\text{Case} & (a,b,c) & \text{Realization}  & \bm{m} & \text{Method} \\ 
\hline
(1,2,4) &  (5, 2, 4 ) &  [ 8, 10, 3, 7, 2, 0, 5, 9, 11, 1, 6, 4 ] & ( 9, 3, 4 ) & 4 \\   
&( 1, 2, 9 ) & [ 1, 6, 11, 2, 10, 5, 3, 8, 0, 9, 4, 12, 7 ] & (-, 6, 7) \\ 
&( 1, 6, 4 ) &  [ 6, 10, 3, 8, 0, 4, 9, 5, 1, 11, 7, 2 ] & ( -, 3, 6)   \\  
&( 3, 2, 9 )&  [ 8, 13, 3, 14, 12, 7, 2, 4, 9, 11, 1, 6, 10, 5, 0 ] & (-, 7, 8) \\  
&(  3, 6, 4 )& [ 0, 12, 2, 7, 11, 1, 10, 6, 4, 9, 5, 3, 13, 8 ] & ( -,7, 8) \\ 
\hline
(1,2,5) & ( 3, 2, 5 ) & [ 5, 0, 7, 2, 9, 3, 8, 10, 1, 6, 4 ] & (8, 3, 4) & 3 \\ 
&( 1, 2, 10 )& [ 13, 9, 0, 5, 10, 1, 6, 11, 2, 4, 8, 3, 12, 7 ] & (-, 3, 7 )  \\   
&(  1, 6, 5 )&  [ 6, 10, 1, 5, 7, 2, 11, 3, 8, 12, 4, 9, 0 ] & ( -, 4, 6) \\  
\hline
(1,3,1) & ( 3, 7, 6 ) &  [ 10, 15, 11, 6, 2, 7, 12, 8, 4, 0, 13, 1, 3, 5, 9, 14, 16 ] & ( 1, 7, 10 ) & 5 \\  
&( 5,  3, 1 ) & [ 2, 6, 4, 8, 0, 5, 7, 3, 1, 9 ] & (-,-, 2) \\  
&( 1,  7, 1 )&  [ 5, 1, 7, 3, 9, 4, 0, 8, 2, 6 ] & (-,3,-) \\  
&( 1,  3, 6 )& [ 1, 6, 10, 4, 9, 5, 3, 8, 2, 7, 0 ] & (-,-, 5) \\   
&( 1, 7, 6 )& [ 8, 12, 7, 3, 14, 4, 9, 5, 0, 10, 6, 2, 13, 11, 1 ] &(-, 5, 8) \\ 
&( 3,  3, 6 )& [ 5, 0, 8, 6, 1, 10, 12, 4, 9, 11, 2, 7, 3 ] & ( 9,-, 4 ) \\  
&( 3, 7, 1 )& [ 9, 7, 3, 11, 1, 5, 0, 8, 4, 6, 2, 10 ] & ( 8, 3,- ) \\ 
\hline
(1,3,2) & ( 3, 7, 7 ) & [ 11, 6, 2, 16, 0, 5, 9, 4, 17, 3, 7, 12, 10, 14, 1, 15, 13, 8 ] & (14, 5, 8) & 5 \\  
&( 5,  3, 2 ) &[ 1, 3, 8, 10, 5, 7, 0, 9, 2, 6, 4 ] & ( 2,-,-) \\   
&( 1, 7, 2 )&  [ 1, 10, 6, 2, 8, 4, 0, 7, 3, 9, 5 ] & (-,5,-) \\   
&( 1,  3, 7)& [ 10, 5, 0, 8, 3, 1, 6, 2, 9, 4, 11, 7 ] & ( -,-,4) \\  
&( 1,  7, 7 )& [ 0, 4, 8, 3, 5, 10, 15, 11, 6, 2, 14, 9, 13, 1, 12, 7 ] & (-, 5, 7) \\  
&( 3,  3, 7 )& [ 4, 6, 1, 11, 2, 7, 5, 0, 10, 12, 3, 8, 13, 9 ] & (9,-, 4) \\  
&( 3,  7, 2 )& [ 10, 12, 1, 6, 2, 11, 3, 7, 5, 9, 0, 4, 8 ] & (10, 3,-) \\ 
\hline
(1,3,3) & ( 1, 11, 3 ) &  [ 7, 3, 15, 11, 0, 5, 9, 13, 1, 6, 2, 14, 10, 12, 8, 4 ] & (9, 3, 4) & 4 \\  
&( 1, 3, 8 )& [ 10, 2, 6, 1, 9, 11, 3, 7, 12, 4, 8, 0, 5 ] & ( 8, -, 4 ) \\
&(5, 3, 3 )& [ 5, 10, 0, 2, 7, 3, 1, 9, 11, 6, 4, 8 ] & ( 1,-, 6 ) \\ 
&( 1, 7, 3 )& [ 7, 3, 11, 4, 8, 10, 2, 6, 1, 9, 5, 0 ] & ( 9,-, 4 ) \\ 
\hline
(1,3,4) & ( 3, 3, 4 ) &  [ 4, 8, 10, 3, 5, 0, 9, 2, 7, 1, 6 ] & ( 8, 3, 4 ) & 3 \\
&( 1, 3, 9 )& [ 5, 7, 2, 12, 3, 8, 13, 4, 9, 0, 10, 1, 11, 6 ] &(-, 5, 6) \\ 
&( 1, 7, 4 )& [ 7, 12, 4, 8, 3, 11, 2, 0, 9, 5, 1, 10, 6 ] & (-, 6, 7 ) \\
\hline
(1,3,5) & (3, 3, 5 ) &  [ 7, 2, 10, 3, 8, 0, 5, 9, 11, 1, 6, 4 ] & ( 9, 3, 4  ) & 3 \\  
&(1, 3, 10 )& [ 0, 10, 5, 1, 11, 6, 4, 9, 14, 3, 8, 13, 2, 12, 7 ] & (-,6, 7)  \\  
&( 1, 7, 5 )& [ 4, 9, 13, 8, 10, 0, 5, 1, 11, 6, 2, 12, 3, 7 ]  & ( -, 3, 4 ) \\  
\hline
(1,4,1) & ( 1, 4, 11 )& [ 5, 9, 4, 16, 11, 7, 2, 14, 12, 0, 13, 1, 6, 10, 15, 3, 8 ] &( 11, 4, 5 ) & 4 \\ 
&(1, 8, 1 )& [ 6, 2, 9, 0, 7, 3, 10, 4, 8, 1, 5 ] & ( 8, 4, -)  \\   
&( 5, 4, 1 )& [ 4, 6, 8, 1, 10, 3, 7, 2, 0, 9, 5 ] & ( 3, 5, - )  \\  
&( 1, 4, 6 )& [ 5, 10, 2, 4, 9, 1, 6, 11, 7, 0, 8, 3 ] & ( 2, 6, -)  \\ 
\hline
(1,4,2) & (1, 8, 2 ) &  [ 5, 9, 1, 6, 2, 10, 0, 8, 4, 11, 3, 7 ] & ( 9, 4, 5 ) & 3 \\ 
&( 1, 4, 7 )&  [ 5, 0, 9, 4, 12, 8, 6, 1, 10, 2, 7, 3, 11 ] & ( 9,-, 4) \\  
&( 5, 4, 2 )& [ 4, 8, 3, 5, 7, 11, 9, 1, 6, 2, 0, 10 ] & ( 1,-, 5)  \\   
\hline
(1,4,3) & ( 3, 4, 3 ) &  [ 4, 8, 2, 6, 1, 10, 3, 7, 9, 0, 5 ] & (  8, 3, 4) & 3 \\   
&( 1, 4, 8 )& [ 7, 11, 1, 6, 10, 5, 0, 9, 4, 2, 12, 3, 8, 13 ] & (-,3, 7) \\   
&(  1, 8, 3)& [ 6, 10, 5, 1, 9, 11, 2, 7, 3, 12, 8, 4, 0 ] & (-,3, 6) \\  
\hline
(1,4,4) &  (1, 4, 9 ) &  [ 4, 6, 1, 11, 7, 2, 12, 8, 3, 13, 9, 14, 10, 0, 5 ] & ( 8, 3, 4 ) & 2 \\  
&( 1, 4, 4 ) &  [ 5, 1, 6, 2, 7, 3, 8, 0, 4, 9 ] & ( 1, 3, -) \\  
\hline

\hline
\end{array}$$
\end{footnotesize}
\end{center}
\end{table}

\begin{table}[tp]
\caption{Growable  realizations for $\{2^a, 4^b,5^c\}$ separated by congruence modulo~$(2,4,5)$: case $(1,4,5)$ to $(2,3,2)$.
They are $x$-growable at $m_x$ in $\bm{m} = (m_2,m_4,m_5)$.}\label{tab:245c}
\begin{center}
\begin{footnotesize}
$$\begin{array}{llllll}
\hline
\text{Case} & (a,b,c) & \text{Realization}  & \bm{m} & \text{Method} \\ 
\hline
(1,4,5) & ( 1, 4, 10 )&  [ 7, 2, 13, 9, 4, 15, 11, 0, 12, 1, 6, 8, 3, 14, 10, 5 ] & (10, 4, 5 ) & 2 \\ 
&( 1, 4, 5 ) & [ 3, 7, 2, 8, 1, 5, 10, 6, 0, 9, 4 ] & ( 2, 4,-) \\  
\hline
(2,1,1) & ( 4, 5, 6 ) &  [ 9, 13, 1, 12, 7, 2, 14, 0, 4, 15, 3, 5, 10, 8, 6, 11 ] & (13, 8, 9 ) & 5 \\ 
&( 8,  1, 1) &  [ 10, 1, 8, 6, 4, 2, 0, 9, 3, 5, 7 ] & (3,-,-) \\   
&( 2,  9, 1 )& [ 7, 3, 12, 8, 4, 0, 2, 6, 11, 9, 5, 1, 10 ] & (-, 3,-) \\   
&( 2,  5, 6 )& [ 0, 5, 1, 10, 12, 8, 3, 13, 9, 4, 2, 7, 11, 6 ] & ( -, 9, 4 ) \\  
&( 4,  1, 6 )& [ 3, 7, 2, 9, 4, 11, 1, 6, 8, 10, 0, 5 ] & (9,-, 4 ) \\
&(  4,  5, 1 )&  [ 3, 7, 0, 2, 6, 8, 10, 1, 5, 9, 4 ] & ( 1, 4,-)  \\
\hline
(2,1,2) & ( 2, 1, 12) &  [ 4, 9, 14, 0, 11, 15, 10, 5, 3, 8, 13, 2, 7, 12, 1, 6 ] & ( 2, 10, 5 ) & 4 \\ 
&( 2, 5, 2 )& [ 5, 1, 7, 3, 8, 6, 2, 0, 4, 9 ] &( 7, 3,-) \\   
&(8, 1, 2 )& [ 7, 9, 11, 1, 6, 8, 4, 2, 0, 10, 5, 3 ] & ( 1, 7,-)  \\  
&(2, 5, 7 )& [ 7, 9, 14, 3, 13, 8, 4, 2, 12, 1, 5, 10, 0, 11, 6 ] & ( 5, 7,-) \\   
&(4, 1, 7 )& [ 6, 8, 3, 11, 0, 2, 10, 1, 9, 4, 12, 7, 5 ]  & (  4, 6,-)  \\   
\hline
(2,1,3) & (2, 5, 3 ) &  [ 4, 8, 10, 3, 7, 1, 6, 2, 9, 0, 5 ]  & (8, 3, 4 ) & 3 \\
&(2, 1, 8 ) & [ 11, 4, 6, 1, 8, 3, 10, 0, 5, 9, 2, 7 ]  & (9,-, 4) \\ 
&( 6, 1, 3 )& [ 3, 8, 6, 4, 2, 7, 9, 0, 5, 1, 10 ]  & ( 1,-, 4) \\  
\hline
(2,1,4) & (2, 5, 4 ) &  [ 6, 1, 9, 11, 7, 2, 10, 3, 5, 0, 8, 4 ]  & (8, 3, 4) & 4 \\ 
&( 2, 1, 9 )& [ 3, 8, 12, 7, 2, 4, 9, 1, 6, 11, 0, 5, 10 ] & (1,-, 4) \\ 
&( 6, 1, 4 )& [ 11, 1, 6, 2, 4, 9, 7, 5, 0, 10, 8, 3 ]  & (1,-, 4) \\  
\hline
(2,1,5) & (2, 5, 5 ) &  [ 4, 8, 3, 12, 10, 6, 1, 9, 5, 0, 11, 2, 7 ]  & (10, 3, 4) & 3 \\  
&( 2, 1, 10 )& [ 5, 10, 1, 6, 11, 2, 4, 9, 0, 12, 7, 3, 8, 13 ] & ( 2,-, 7) \\   
&(  4, 1, 5 )& [ 8, 2, 7, 9, 3, 5, 0, 4, 6, 1, 10 ]  & ( 8,-, 4) \\  
\hline
(2,2,1) & (2, 2, 11) & [ 7, 2, 13, 8, 3, 5, 10, 15, 11, 6, 1, 12, 0, 14, 9, 4 ] &( 2, 11, 5) & 4 \\ 
&( 2, 6, 1 )& [ 6, 2, 0, 8, 4, 9, 5, 1, 7, 3 ] &(1, 4,-) \\ 
&( 6, 2, 1 )& [ 5, 7, 3, 1, 9, 4, 6, 2, 0, 8 ]  &(1, 5,-) \\ 
&(2, 2, 6 )& [ 2, 7, 1, 3, 8, 10, 6, 0, 5, 9, 4 ] &( 1, 5,-) \\   
\hline
(2,2,2) & ( 2, 6, 2) & [ 6, 2, 9, 3, 7, 0, 5, 1, 10, 8, 4 ]  &( 8, 3, 4) & 3\\ 
&(  2, 2, 7 )&  [ 4, 9, 1, 3, 8, 10, 5, 0, 7, 2, 6, 11 ]  &( 1,-, 6) \\ 
&( 8, 2, 2 ) & [ 9, 11, 0, 5, 1, 12, 10, 8, 6, 4, 2, 7, 3 ] &(1,- 4) \\  
\hline
(2,2,3) & (2, 6, 3 ) & [ 6, 1, 8, 4, 0, 10, 2, 7, 3, 11, 9, 5 ]  &(9, 3, 5 ) & 3 \\
&( 2, 2, 8 ) & [ 7, 2, 11, 3, 8, 6, 1, 9, 4, 0, 5, 10, 12 ] & ( 10,-, 4) \\
&( 6, 2, 3 ) & [ 3, 5, 0, 10, 6, 1, 11, 9, 7, 2, 4, 8 ] &( 1,-, 4) \\  
\hline
(2,2,4) & (2, 6, 4 ) &  [ 8, 3, 12, 4, 9, 0, 11, 2, 7, 5, 1, 10, 6 ] & ( 10, 4, 6 ) & 3 \\  
&( 2, 2, 9 ) & [ 3, 8, 12, 7, 2, 4, 9, 13, 11, 6, 1, 10, 5, 0 ] &( 1,-, 4) \\ 
&( 4, 2, 4 )& [ 3, 8, 10, 1, 5, 0, 6, 2, 4, 9, 7 ] &(1,-, 4 ) \\  
\hline
(2,2,5) & ( 2, 2, 10 ) &  [ 6, 11, 1, 12, 2, 7, 9, 4, 14, 10, 5, 0, 13, 3, 8 ] &( 11, 5, 6) & 3 \\  
&( 2, 6, 5 ) & [ 12, 0, 5, 1, 10, 6, 8, 4, 9, 13, 3, 7, 2, 11 ]  &( 10, 3,-) \\  
&( 4, 2, 5) & [ 0, 10, 3, 8, 1, 5, 7, 9, 2, 6, 4, 11 ] &( 8, 4, -) \\ 
\hline
(2,3,1) & ( 2, 3, 11 ) &  [ 4, 8, 3, 15, 10, 5, 0, 13, 1, 6, 11, 16, 12, 14, 2, 7, 9 ] &(11, 3, 4) & 4 \\ 
&( 2, 7, 1 ) & [ 5, 9, 2, 4, 8, 1, 10, 3, 7, 0, 6 ] & ( 3, 6,-) \\ 
&( 6, 3, 1 ) & [ 10, 1, 6, 8, 4, 2, 0, 9, 5, 3, 7 ] & (1, 6,-) \\  
&(  2, 3, 6 ) &  [ 11, 4, 6, 10, 3, 5, 1, 8, 0, 7, 2, 9 ] & (7, 4,-) \\  
\hline
(2,3,2) & ( 2, 3, 12 ) &  [ 3, 8, 13, 17, 12, 7, 2, 4, 9, 14, 10, 15, 1, 6, 11, 16, 0, 5 ] & ( 1, 9, 4) & 4 \\  
&( 2, 11, 2 ) &  [ 6, 2, 13, 1, 5, 9, 7, 11, 15, 3, 8, 4, 0, 12, 14, 10 ] & ( 11, 3,-) \\  
&( 4, 3, 2 ) &  [ 0, 5, 7, 3, 1, 9, 4, 8, 6, 2 ] & (2, 5,-) \\ 
&(  2, 3, 7 ) &  [ 6, 1, 9, 0, 8, 4, 12, 10, 2, 7, 11, 3, 5 ]  &(8, 4,-) \\  
\hline
\end{array}$$
\end{footnotesize}
\end{center}
\end{table}

\begin{table}[tp]
\caption{Growable  realizations for $\{2^a, 4^b,5^c\}$ separated by congruence modulo~$(2,4,5)$: case $(2,3,3)$ to $(2,4,5)$.
They are $x$-growable at $m_x$ in $\bm{m} = (m_2,m_4,m_5)$.}\label{tab:245d}
\begin{center}
\begin{footnotesize}
$$\begin{array}{llllll}
\hline
\text{Case} & (a,b,c) & \text{Realization}  & \bm{m} & \text{Method} \\ 
\hline
(2,3,3) & ( 2, 3, 8 ) &  [ 6, 10, 1, 11, 13, 4, 9, 0, 5, 7, 2, 12, 3, 8 ] & (10, 5, 6) & 3 \\  
&( 2, 7, 3 ) & [ 6, 8, 4, 12, 3, 7, 11, 2, 0, 9, 1, 10, 5 ] & (4, 6,-) \\  
&( 4, 3, 3 ) & [ 1, 6, 10, 8, 3, 5, 9, 4, 2, 0, 7 ] & (1, 6,-) \\  
\hline
(2,3,4) & (2, 3, 9 ) &  [ 7, 2, 12, 14, 10, 5, 0, 11, 1, 6, 8, 3, 13, 9, 4 ] &(10, 3, 4) & 2 \\
&(2, 3, 4 ) & [ 1, 6, 2, 0, 5, 7, 3, 8, 4, 9 ] &(3, 5,-) \\  
\hline
(2,3,5) & (2, 3, 10) &  [ 6, 11, 15, 4, 9, 7, 2, 13, 1, 12, 8, 3, 14, 0, 5, 10 ] & (12, 5, 6) & 2 \\ 
&( 2, 3, 5 )&  [ 6, 0, 2, 8, 4, 10, 1, 7, 3, 9, 5 ]  &(4, 6,-) \\ 
\hline
(2,4,1) & ( 2, 4, 11 ) &  [ 6, 11, 16, 12, 7, 2, 15, 13, 0, 14, 1, 5, 10, 8, 3, 17, 4, 9 ]  &( 12, 4, 6) & 4 \\  
&( 2, 8, 1 ) &  [ 3, 7, 11, 9, 1, 5, 10, 6, 2, 0, 8, 4 ]  &( 1, 4,-) \\  
&( 4, 4, 1 )&  [ 5, 7, 3, 9, 1, 6, 2, 0, 8, 4 ]  &( 3, 5,-) \\ 
&(  2, 4, 6 )& [ 12, 4, 6, 2, 10, 1, 9, 0, 8, 3, 11, 7, 5 ]  &(  8, 5,-) \\   
\hline
(2,4,2) &  (2, 4, 7) &  [ 4, 8, 13, 3, 7, 12, 10, 1, 6, 2, 11, 9, 0, 5 ]  &(8, 3, 4) & 3 \\  
&( 2, 8, 2 )& [ 11, 3, 7, 5, 9, 0, 4, 8, 12, 10, 1, 6, 2 ]  &(9, 3,-)  \\  
&( 4, 4, 2 )& [ 8, 10, 3, 7, 1, 5, 0, 9, 2, 4, 6 ]  &(8, 3,-) \\ 
\hline
(2,4,3) & (4, 4, 3) &  [ 11, 1, 3, 7, 2, 4, 9, 5, 0, 8, 10, 6 ]  &(1, 5, 6) & 2 \\ 
&(2, 4, 3 ) &  [ 3, 7, 2, 6, 1, 9, 5, 0, 8, 4 ]  &( 3, 5,-) \\   
\hline
(2,4,4) & ( 2, 4, 9 ) & [ 10, 5, 0, 4, 8, 13, 15, 3, 14, 2, 7, 12, 1, 6, 11, 9 ] &(13, 7, 9 ) & 2 \\  
&(2,  4, 4 ) & [ 9, 3, 5, 10, 6, 1, 8, 2, 4, 0, 7 ] & (6, 3,-) \\  
\hline
(2,4,5) &  (2, 4, 5 ) & [ 5, 0, 8, 3, 10, 2, 7, 11, 9, 1, 6, 4 ] &( 8, 3, 4) & 1 \\ 
\hline
\end{array}$$
\end{footnotesize}
\end{center}
\end{table}

\begin{table}[tp]
\caption{Growable  realizations for $\{3^a, 4^b,5^c\}$ separated by congruence modulo~$(3,4,5)$: case $(1,1,1)$ to $(1,3,2)$.
They are $x$-growable at $m_x$ in $\bm{m} = (m_3,m_4,m_5)$.}\label{tab:345a}
\begin{center}
\begin{footnotesize}
$$\begin{array}{llllll}
\hline
\text{Case} & (a,b,c) & \text{Realization}  & \bm{m} & \text{Method} \\ 
\hline
(1,1,1) & (4,5,6) & [ 5, 10, 14, 2, 7, 4, 15, 11, 8, 3, 0, 12, 1, 6, 9, 13 ] &  ( 10, 4, 5 ) & 5 \\
&(1, 1, 11) &  [ 12, 7, 2, 11, 6, 1, 10, 5, 0, 3, 8, 13, 9, 4 ] & (-, -, 4 ) \\
&( 7, 1, 1 ) &  [ 2, 5, 9, 6, 3, 8, 1, 4, 7, 0 ] &  (3,-,-) \\
&( 1, 5, 6 ) & [ 6, 11, 7, 2, 10, 1, 5, 9, 4, 12, 8, 3, 0 ] &  (-, 4, 6 ) \\
&( 4, 1, 6 ) &  [ 9, 0, 5, 2, 10, 7, 4, 11, 6, 1, 8, 3 ] & (8,-, 4) \\
&( 4, 5, 1 ) &  [ 5, 9, 6, 2, 10, 3, 8, 1, 4, 7, 0 ] & (3, 5,-) \\
\hline
(1,1,2) & (1, 9, 2 ) &  [ 5, 9, 1, 10, 6, 2, 11, 7, 3, 0, 4, 12, 8 ] & (7, 8,-) & 3 \\
&(7, 1, 2 ) &  [ 4, 7, 10, 5, 2, 9, 1, 6, 3, 0, 8 ] &  (2, 4,- ) \\
&(4, 1, 7 ) &  [ 4, 9, 1, 6, 3, 12, 7, 2, 10, 0, 5, 8, 11 ] & (9, 3, 4 ) \\
&(1, 5, 7 ) & [ 4, 8, 13, 9, 0, 5, 2, 12, 3, 7, 11, 6, 1, 10 ] & ( 8, 3, 4 ) \\
\hline
(1,1,3) & ( 10, 1, 3 ) & [ 6, 9, 12, 0, 3, 8, 11, 14, 2, 7, 4, 1, 5, 10, 13 ] & (10, 4, 6 ) & 4 \\
&(1, 1, 8 ) &  [ 5, 10, 4, 9, 1, 6, 0, 7, 2, 8, 3 ] & (-, 4, 5 ) \\
&( 1, 9, 3 ) &  [ 6, 10, 13, 3, 7, 2, 11, 1, 5, 9, 0, 4, 8, 12 ] &  ( -, 3, 6)  \\
&( 4, 1, 8 ) & [ 12, 3, 6, 11, 1, 4, 9, 0, 5, 10, 7, 2, 13, 8 ] & (-, 3, 8) \\ 
&( 4, 5, 3 ) & [ 3, 8, 4, 0, 10, 5, 1, 11, 2, 6, 9, 12, 7 ] & (-, 3, 7 )   \\
&( 7, 1, 3 ) &  [ 10, 1, 4, 9, 0, 3, 7, 2, 5, 8, 11, 6 ] & (-,  5, 6)  \\
\hline
(1,1,4) & ( 7, 1, 4 ) & [ 11, 1, 6, 9, 12, 8, 3, 0, 5, 2, 10, 7, 4 ] & ( 8, 3, 4 ) & 4 \\
&( 1, 1, 9 ) &  [ 6, 11, 7, 2, 9, 4, 1, 8, 3, 10, 5, 0 ] & (-, 3, 6)\\
&( 1, 5, 4 ) &  [ 6, 1, 8, 0, 4, 9, 2, 7, 3, 10, 5 ] & (-, 3, 5)  \\
&( 4, 1, 9 ) & [ 4, 9, 12, 7, 2, 14, 10, 5, 0, 3, 8, 13, 1, 11, 6 ] & (-,5, 6) \\
&( 4, 5, 4 ) & [ 6, 10, 5, 1, 4, 9, 13, 2, 7, 11, 0, 3, 12, 8 ] & (-, 7, 8 ) \\
\hline
(1,1,5) & ( 4, 1, 5 ) &  [ 5, 10, 7, 2, 8, 0, 4, 9, 1, 6, 3 ] & (2, 3, 5) & 3 \\
&( 1, 1, 10 ) & [ 5, 0, 10, 2, 7, 12, 8, 3, 11, 6, 1, 9, 4 ] & (-, 3, 4 ) \\
&( 1, 5, 5 ) &  [ 5, 9, 4, 1, 6, 2, 10, 3, 8, 0, 7, 11 ] & (-,4, 5) \\   
\hline
(1,2,1) & ( 4, 6, 6 ) & [ 5, 10, 6, 2, 15, 1, 14, 11, 16, 4, 9, 12, 7, 3, 0, 13, 8 ] & ( 4, 5, 8 ) & 5 \\
&( 7, 2, 1 ) &  [ 7, 10, 2, 8, 5, 1, 4, 0, 3, 6, 9 ] & ( 7,-,-)  & \\
&( 1, 10, 1) &  [ 4, 8, 12, 7, 3, 0, 9, 5, 1, 10, 6, 2, 11 ] & (-, 4,-)  \\  
&( 1, 6, 6 ) &  [ 6, 10, 0, 5, 9, 13, 4, 8, 3, 12, 1, 11, 2, 7 ] & (-, 5, 6) \\ 
&( 4, 2, 6 )  &  [ 5, 8, 3, 12, 4, 7, 2, 10, 0, 9, 1, 11, 6 ] &  (4,-, 6 )  \\ 
&( 4, 6, 1 )  &  [ 10, 6, 3, 0, 4, 8, 5, 1, 9, 2, 11, 7 ] & (5, 7,- ) \\ 
\hline
(1,2,2) & ( 1, 2, 7 ) & [ 4, 9, 2, 7, 1, 6, 10, 5, 0, 8, 3 ] & ( 2, 3, 4 ) & 3 \\ 
&( 1, 6, 2 ) &  [ 8, 2, 6, 1, 5, 9, 4, 0, 7, 3 ] & ( 2, 3, - ) \\  
&( 7, 2, 2 ) &  [ 5, 2, 11, 3, 8, 0, 9, 6, 1, 10, 7, 4 ] & (3, 4, -) \\   
\hline
(1,2,3) & ( 1, 2, 8 ) &  [ 5, 10, 1, 6, 11, 3, 8, 0, 7, 2, 9, 4 ] & ( 3, 4, 5 ) & 3 \\ 
&( 1, 6, 3 ) &  [ 8, 1, 5, 9, 4, 0, 6, 10, 2, 7, 3 ] & (  2, 3 ,- ) \\ 
&( 4, 2, 3 ) &  [ 8, 5, 0, 3, 7, 4, 9, 2, 6, 1 ] & ( 4, 5,-) \\   
\hline
(1,2,4) & ( 1, 2, 9 ) &  [ 4, 9, 12, 7, 2, 11, 6, 1, 10, 5, 0, 8, 3 ] & ( 2, 3, 4 ) & 3 \\   
&( 1, 6, 4 ) & [ 4, 8, 0, 5, 2, 9, 1, 6, 10, 3, 7, 11 ] & ( 8, 3, -) \\   
&( 4, 2, 4) &  [ 3, 8, 0, 5, 9, 1, 6, 10, 2, 7, 4 ] &  (2, 3, -) \\
\hline
(1,2,5) & ( 4, 2, 5 ) &  [ 6, 9, 4, 11, 2, 7, 3, 0, 8, 1, 10, 5 ] & ( 4, 5, 6) & 3 \\  
&( 1, 2, 10 )&  [ 8, 11, 2, 7, 12, 3, 13, 4, 9, 0, 5, 10, 6, 1 ] & ( -, 7, 8) \\ 
&( 1, 6, 5 ) &  [ 12, 3, 8, 4, 0, 5, 9, 1, 11, 7, 2, 10, 6 ] & (-, 5, 6) \\  
\hline
(1,3,1) & ( 1, 3, 6) & [ 5, 10, 6, 1, 7, 2, 9, 4, 0, 8, 3 ] & (2, 3, 5) & 3 \\   
& ( 1, 7, 1) & [ 5, 9, 4, 8, 2, 6, 0, 3, 7, 1 ] & (4, 5, -) \\ 
& ( 7, 3, 1) &  [ 4, 7, 10, 1, 5, 9, 0, 3, 8, 11, 2, 6 ] & ( 2, 4, -) \\  
\hline
(1,3,2) & ( 1, 3, 7 ) &  [ 5, 10, 2, 7, 0, 4, 9, 1, 6, 11, 8, 3 ] & (  2, 3, 5 ) & 3 \\ 
&( 1, 7, 2 ) &  [ 4, 9, 2, 6, 10, 5, 1, 8, 0, 7, 3 ] & (  3, 4,-) \\ 
&( 7, 1, 2) &  [ 3, 6, 9, 1, 4, 0, 8, 5, 10, 2, 7 ] & (2, 3,-) \\  
\hline
\end{array}$$
\end{footnotesize}
\end{center}
\end{table}

\begin{table}[tp]
\caption{Growable realizations for $\{3^a, 4^b,5^c\}$ separated by congruence modulo~$(3,4,5)$: case $(1,3,3)$ to $(2,2,1)$.
They are $x$-growable at $m_x$ in $\bm{m} = (m_3,m_4,m_5)$.}\label{tab:345b}
\begin{center}
\begin{footnotesize}
$$\begin{array}{llllll}
\hline
\text{Case} & (a,b,c) & \text{Realization}  & \bm{m} & \text{Method} \\ 
\hline
(1,3,3) & ( 1, 3, 8 ) &  [ 6, 11, 2, 7, 12, 3, 8, 0, 10, 5, 1, 9, 4 ] & ( 3, 4, 6 )  & 3 \\  
& ( 1, 7, 3 ) &  [ 9, 0, 4, 8, 1, 5, 10, 2, 6, 11, 7, 3 ] & ( 2, 3,-) \\  
& ( 4, 3, 3 ) & [ 4, 7, 0, 8, 3, 9, 1, 5, 10, 2, 6 ] & ( 3, 4,- ) \\  
\hline
(1,3,4) & ( 4, 3, 4 ) &  [ 4, 7, 0, 5, 1, 9, 6, 2, 11, 8, 3, 10 ] & ( 8, 3, 4  ) & 3 \\  
& ( 1, 3, 9 ) &  [ 10, 5, 1, 6, 11, 2, 7, 12, 3, 0, 4, 9, 13, 8 ] & ( -, 7, 8 ) \\  
& ( 1, 7, 4 ) &  [ 4, 8, 12, 7, 3, 11, 2, 6, 1, 10, 0, 5, 9 ] & ( -, 3, 4 ) \\ 
\hline
(1,3,5) & (1, 3, 10 ) &  [ 7, 12, 2, 6, 11, 1, 13, 8, 3, 14, 9, 4, 0, 10, 5 ] & ( 4, 5, 7) & 3 \\
& (1, 7, 5 ) &  [ 6, 11, 1, 5, 10, 0, 9, 12, 2, 7, 3, 13, 8, 4 ] & ( 3, 4, - ) \\ 
& ( 4, 3, 5 ) &  [ 11, 2, 5, 0, 8, 3, 12, 9, 1, 6, 10, 7, 4 ] & ( 3, 4, - ) \\ 
\hline
(1,4,1) & ( 1, 4, 11 ) &  [ 10, 15, 11, 16, 3, 7, 12, 0, 5, 8, 13, 1, 6, 2, 14, 9, 4 ] & ( 9, 3, 4 ) & 4 \\ 
& ( 1, 8, 1 ) &  [ 6, 0, 7, 3, 10, 2, 9, 5, 1, 8, 4 ] & (   5, 6, -   )  \\ 
& ( 4, 4, 1 ) &  [ 6, 0, 4, 7, 3, 8, 1, 5, 2, 9 ] & ( 2, 3,- ) \\  
& ( 1, 4, 6 ) &  [ 4, 9, 1, 6, 11, 8, 0, 5, 10, 2, 7, 3 ] & ( 2, 3,-)  \\ 
\hline
(1,4,2) & ( 1, 4, 7 ) &  [ 5, 9, 1, 6, 2, 10, 0, 8, 4, 12, 7, 3, 11 ] & ( 9, 4, 5 ) & 3 \\  
& ( 1, 8, 2 ) &  [ 11, 3, 7, 10, 6, 2, 9, 1, 5, 0, 8, 4 ] &  (8, 3, - ) \\  
& ( 4, 4, 2 ) & [ 5, 8, 0, 7, 3, 10, 2, 6, 1, 9, 4 ] & ( 3, 4, -) \\ 
\hline
(1,4,3) & ( 1, 4, 8 ) &  [ 4, 7, 12, 3, 8, 13, 9, 5, 0, 10, 1, 6, 2, 11 ] & ( 9, 3, 4 ) & 3 \\  
& ( 1, 8, 3 ) & [ 5, 9, 0, 8, 4, 12, 3, 7, 2, 11, 1, 10, 6 ] & (4, 5, - ) \\  
& ( 4, 4, 3 ) &  [ 4, 8, 11, 3, 7, 0, 5, 2, 10, 1, 6, 9 ] & ( 8, 3, - ) \\  
\hline
(1,4,4) & ( 1, 4, 9 ) & [ 10, 0, 5, 9, 14, 4, 8, 13, 3, 6, 1, 12, 2, 7, 11 ] & ( 9, 10, 4) & 3 \\ 
& ( 1, 8, 4 ) & [ 7, 11, 1, 5, 10, 0, 4, 9, 6, 2, 12, 3, 13, 8 ] &  (6, 7, -) \\  
& ( 4, 4, 4 ) &  [ 4, 9, 0, 10, 6, 1, 11, 2, 5, 8, 3, 12, 7 ] & ( 3, 4, -) \\ 
\hline
(1,4,5) & ( 1, 4, 5 ) &  [ 4, 9, 5, 0, 7, 2, 10, 6, 1, 8, 3 ] & (2, 3, 4) & 1 \\  
\hline
(2,1,1) & ( 5, 5, 6 ) & [ 7, 12, 8, 4, 16, 3, 0, 13, 9, 6, 1, 15, 10, 5, 2, 14, 11 ] & (6, 7, 11) & 5 \\ 
& ( 8, 1, 1 ) &  [ 10, 2, 5, 8, 0, 4, 7, 1, 9, 6, 3 ] & ( 4,-,-)  & \\  
& ( 2, 9, 1 ) &  [ 9, 5, 1, 6, 10, 0, 4, 8, 12, 2, 11, 7, 3 ] & (-, 6,-) \\ 
& ( 2, 5, 6 ) &  [ 9, 0, 5, 2, 11, 1, 6, 10, 13, 4, 8, 12, 7, 3 ] & ( -,  9, 4) \\   
& ( 5, 1, 6 ) & [ 0, 5, 2, 10, 7, 4, 1, 6, 11, 3, 8, 12, 9 ] & ( 8,-, 4)   \\
& ( 5, 5, 1 ) & [ 1, 5, 8, 4, 0, 9, 2, 11, 3, 6, 10, 7 ] & ( 5, 7, -) \\ 
\hline
(2,1,2) &  ( 2, 1, 12) & [ 5, 10, 15, 4, 7, 2, 13, 8, 3, 14, 9, 6, 1, 12, 0, 11 ] & (10, 4, 5) & 4 \\
& ( 2, 5, 2  ) & [ 0, 5, 9, 3, 6, 2, 7, 1, 4, 8 ] & ( 3, 5, -)  \\ 
& ( 8, 1, 2 ) &  [ 9, 0, 3, 6, 2, 11, 4, 7, 10, 1, 8, 5 ] & (  4, 5,-)  \\  
& ( 2, 1, 7 ) &  [ 2, 7, 1, 8, 0, 6, 3, 9, 4, 10, 5 ] &  (4, 5,-)  \\   
\hline
(2,1,3) & ( 2, 5, 3 ) &  [ 3, 8, 4, 0, 7, 2, 10, 6, 1, 9, 5 ] &  ( 2, 3, 5 ) & 3  \\  
& ( 2, 1, 13) &  [ 15, 10, 5, 0, 12, 16, 11, 6, 3, 8, 13, 1, 4, 9, 14, 2, 7 ] & ( 2,-,  6 ) \\
& ( 8, 1, 3 ) &  [ 9, 12, 2, 5, 0, 3, 8, 11, 1, 6, 10, 7, 4 ] & ( 8,-, 4 )   \\
\hline
(2,1,4) & ( 8, 1, 4  ) & [ 6, 9, 13, 10, 5, 2, 7, 4, 1, 12, 3, 8, 11, 0 ] & (10, 5, 6 ) & 4 \\   
& ( 2, 1, 9 ) &  [ 0, 5, 8, 3, 11, 6, 1, 10, 2, 7, 12, 9, 4 ] & (-, 3, 4) \\ 
& ( 2, 5, 4 ) & [ 5, 9, 2, 6, 1, 10, 3, 8, 0, 4, 7, 11 ] & (-,3, 5) \\
& ( 5, 1, 4 ) &  [ 6, 1, 9, 4, 10, 2, 7, 3, 0, 8, 5 ] & (-, 4, 5)  \\   
\hline
(2,1,5) & ( 2, 1, 10 ) & [ 5, 8, 13, 4, 9, 0, 10, 1, 6, 3, 12, 7, 2, 11 ] &  (9, 4, 5) & 3 \\  
& ( 2, 5, 5 ) & [ 3, 7, 10, 1, 9, 5, 0, 8, 4, 12, 2, 11, 6 ] & ( 5, 6, - ) \\   
& ( 5, 1, 5 ) &  [ 10, 1, 5, 8, 3, 0, 7, 2, 11, 6, 9, 4 ] & (2, 4, - ) \\ 
\hline
(2,2,1) & ( 2, 2, 6 ) & [ 4, 8, 3, 9, 1, 6, 0, 7, 2, 10, 5 ] & ( 3, 4, 5 ) & 3 \\  
& ( 2, 6, 1) &  [ 3, 7, 1, 5, 8, 2, 6, 9, 4, 0 ] &  (2, 3, -)  \\  
& ( 8, 2, 1) &  [ 3, 6, 9, 0, 4, 1, 8, 5, 2, 11, 7, 10 ] & (  2, 3, -) \\
\hline
\end{array}$$
\end{footnotesize}
\end{center}
\end{table}

\begin{table}[tp]
\caption{Growable realizations for $\{3^a, 4^b,5^c\}$ separated by congruence modulo~$(3,4,5)$: case $(2,2,2)$ to $(3,1,4)$.
They are $x$-growable at $m_x$ in $\bm{m} = (m_3,m_4,m_5)$.}\label{tab:345c}
\begin{center}
\begin{footnotesize}
$$\begin{array}{llllll}
\hline
\text{Case} & (a,b,c) & \text{Realization}  & \bm{m} & \text{Method} \\ 
\hline
(2,2,2) & ( 2, 2, 7) &  [ 4, 9, 1, 6, 11, 2, 7, 10, 5, 0, 8, 3 ] & ( 2, 3, 4  ) & 3 \\ 
& ( 2, 6, 2 ) & [ 5, 10, 3, 7, 0, 6, 2, 9, 1, 4, 8 ] & ( 3, 5, -) \\ 
& ( 5, 2, 2 ) &  [ 4, 7, 3, 0, 5, 8, 1, 6, 2, 9 ] & ( 4, 5,- ) \\
\hline
(2,2,3) & ( 2, 2, 8) &  [ 7, 12, 2, 10, 1, 9, 4, 0, 5, 8, 3, 11, 6 ] & ( 5, 6, 7 ) & 3 \\   
& ( 2, 6, 3) &  [ 6, 11, 3, 7, 2, 10, 1, 8, 4, 0, 9, 5 ] & ( 4, 5, - ) \\ 
& ( 5, 2, 3 ) & [ 3, 8, 0, 5, 2, 9, 1, 6, 10, 7, 4 ] & ( 2, 3, - )  \\  
\hline
(2,2,4) & (  2, 2, 9 ) &  [ 7, 12, 2, 11, 1, 10, 5, 0, 3, 8, 13, 4, 9, 6 ] & ( 5, 6, 7 ) & 3 \\   
& ( 2, 6, 4 ) &  [ 1, 9, 0, 5, 8, 4, 12, 3, 11, 2, 6, 10, 7 ] & (6, 7,- ) \\  
& ( 5, 2, 4 ) &  [ 4, 7, 10, 2, 5, 0, 9, 6, 1, 8, 3, 11 ] & ( 8, 3,- ) \\   
\hline
(2,2,5) & ( 2, 2, 10 ) &  [ 10, 0, 5, 8, 13, 3, 7, 12, 2, 6, 1, 11, 14, 9, 4 ] & (9, 3, 4 ) & 3 \\  
& ( 2, 6, 5 ) &  [ 5, 9, 0, 4, 1, 10, 6, 2, 11, 8, 13, 3, 7, 12 ] & (  7, 3, -) \\  
&  ( 5, 2, 5 ) &  [ 4, 0, 10, 2, 7, 12, 9, 1, 5, 8, 11, 6, 3 ] & ( 2, 3,- ) \\  
\hline
(2,3,1) & ( 2, 3, 6 ) &  [ 4, 7, 11, 6, 1, 9, 2, 5, 0, 8, 3, 10 ] & ( 8, 3, 4  ) & 3 \\ 
& ( 2, 7, 1 ) &  [ 4, 1, 8, 5, 0, 7, 3, 10, 6, 2, 9 ] & ( 7, 3,- ) \\  
& ( 5, 3, 1 ) &  [ 6, 9, 2, 5, 1, 8, 3, 7, 4, 0 ] & (2, 3,- )  \\  
\hline
(2,3,2) & ( 2, 3, 7 ) &  [ 4, 9, 1, 6, 2, 10, 0, 5, 8, 12, 7, 3, 11 ] & ( 9, 3, 4 ) & 3 \\  
& ( 2, 7, 2 ) &  [ 5, 10, 2, 6, 9, 1, 8, 4, 0, 3, 7, 11 ] & ( 4, 5,- ) \\   
& ( 5, 3, 2) &  [ 5, 9, 1, 7, 2, 10, 6, 3, 0, 8, 4 ] & (4, 5,- ) \\  
\hline
(2,3,3) & ( 2, 3, 8 ) &  [ 9, 0, 5, 2, 11, 6, 1, 10, 7, 3, 12, 8, 4, 13 ] & ( 8, 9, 4 ) & 3 \\  
& ( 2, 7, 3 ) &  [ 5, 10, 1, 11, 2, 6, 9, 0, 4, 8, 3, 12, 7 ] & ( 4, 5, - ) \\   
& ( 5, 3, 3 ) &  [ 4, 8, 1, 6, 9, 0, 5, 2, 11, 3, 7, 10 ] & ( 8, 3, - ) \\  
\hline
(2,3,4) & ( 2, 3, 9 ) &  [ 6, 11, 1, 12, 2, 7, 10, 14, 9, 4, 0, 5, 8, 3, 13 ] & (  11, 5, 6  ) & 2 \\ 
& ( 2, 3, 4 ) &  [ 1, 6, 2, 7, 3, 0, 5, 8, 4, 9 ] & (4, 5,-)  \\  
\hline
(2,3,5) & ( 2, 3, 5 ) &  [ 5, 10, 7, 2, 9, 1, 6, 0, 4, 8, 3 ] & (  2, 3, 5  ) & 1 \\ 
\hline
(2,4,1) & (2 , 4, 6 ) &  [ 10, 2, 6, 1, 11, 3, 7, 12, 9, 5, 0, 8, 4 ] & ( 9, 3, 4  ) & 3 \\ 
& ( 2 , 8, 1) &  [ 7, 11, 3, 6, 10, 2, 9, 1, 5, 8, 4, 0 ] & ( 5, 7,-)  \\ 
& ( 5, 4, 1 ) &  [ 2, 5, 9, 1, 8, 0, 3, 7, 4, 10, 6 ] & (  4, 6,-) \\  
\hline
(2,4,2) & ( 2, 4, 7 ) & [ 6, 11, 2, 7, 12, 1, 10, 0, 4, 9, 13, 3, 8, 5 ] & (  4, 5, 6 ) & 3 \\ 
& ( 2, 8, 2 ) &  [ 10, 1, 6, 9, 0, 5, 2, 11, 7, 3, 12, 8, 4 ] & ( 8, 3, - ) \\   
& ( 5, 4, 2 ) &  [ 9, 6, 3, 0, 5, 1, 10, 2, 7, 11, 8, 4 ] & ( 8, 3, -  ) \\  
\hline
(2,4,3) & (  2, 4, 8 ) &  [ 9, 14, 4, 7, 11, 6, 1, 5, 10, 0, 3, 13, 2, 12, 8 ] & ( 6, 8, 9 ) & 2 \\  
& ( 2, 4, 3 ) &  [ 7, 1, 5, 0, 4, 9, 2, 6, 3, 8 ] & (2, 3,- ) \\  
\hline
(2,4,4) & ( 2, 4, 4 ) &  [ 3, 8, 2, 6, 10, 7, 0, 5, 1, 9, 4 ] & (  2, 3, 4 ) & 1 \\ 
\hline
(2,4,5) & ( 2, 4, 5 ) &  [ 3, 7, 11, 8, 1, 6, 2, 10, 5, 0, 9, 4 ] & (2, 3, 4 ) & 1 \\ 
\hline
(3,1,1) & ( 3, 1, 6 ) &  [ 5, 10, 2, 7, 1, 6, 9, 4, 0, 8, 3 ] & (2, 3, 5 ) & 3 \\  
& ( 3, 5, 1 ) &  [ 9, 2, 6, 0, 4, 7, 3, 8, 5, 1 ] & (2, 3,-) \\  
& ( 9, 1, 1 ) &  [ 11, 2, 5, 8, 0, 3, 6, 9, 4, 1, 10, 7 ] &  (  2, 7, - )  \\  
\hline
(3,1,2) & ( 3, 1, 7  ) & [ 10, 3, 8, 11, 6, 1, 9, 2, 5, 0, 7, 4 ] & ( 8, 3, 4  ) & 3 \\ 
& ( 3, 5, 2 ) &  [ 4, 7, 0, 6, 2, 10, 3, 8, 1, 9, 5 ] & ( 4, 5,-) \\  
& ( 6, 1, 2 ) & [ 5, 2, 7, 0, 4, 1, 8, 3, 6, 9 ] & ( 2, 3,- ) \\  
\hline
(3,1,3) & ( 3, 1, 8 ) &  [ 4, 7, 12, 9, 1, 6, 10, 2, 5, 0, 8, 3, 11 ] & ( 8, 3, 4 ) & 3 \\  
& ( 3, 5, 3 ) &  [ 4, 9, 5, 1, 10, 6, 11, 8, 0, 3, 7, 2 ] & ( 2, 4, -) \\  
& ( 6, 1, 3 )  &  [ 3, 8, 5, 2, 10, 7, 4, 0, 6, 1, 9 ] & (  2, 3,- )  \\ 
\hline
(3,1,4) & (3, 1, 9 ) &  [ 4, 9, 0, 3, 8, 13, 10, 5, 1, 12, 7, 2, 11, 6 ] & ( 2, 4, 6 ) & 3 \\ 
& ( 3, 5, 4 ) & [ 4, 8, 0, 5, 1, 11, 6, 2, 10, 7, 3, 12, 9 ] & ( 8, 3,-) \\  
& ( 6, 1, 4 ) &  [ 11, 2, 7, 4, 1, 8, 3, 0, 9, 6, 10, 5 ] & (2, 5,-)   \\  
\hline
\end{array}$$
\end{footnotesize}
\end{center}
\end{table}

\begin{table}[tp]
\caption{Growable realizations for $\{3^a, 4^b,5^c\}$ separated by congruence modulo~$(3,4,5)$: case $(3,1,5)$ to $(3,4,5)$.
They are $x$-growable at $m_x$ in $\bm{m} = (m_3,m_4,m_5)$.}\label{tab:345d}
\begin{center}
\begin{footnotesize}
$$\begin{array}{llllll}
\hline
\text{Case} & (a,b,c) & \text{Realization}  & \bm{m} & \text{Method} \\ 
\hline
(3,1,5) & ( 3, 1, 10) &  [ 13, 3, 8, 11, 1, 6, 10, 0, 5, 2, 12, 7, 4, 14, 9 ] & (  8, 9, 4  ) & 3 \\ 
& ( 3, 5, 5 ) &  [ 10, 5, 8, 13, 9, 0, 11, 1, 6, 2, 12, 3, 7, 4 ] & ( 3, 4,-) \\  
& ( 6, 1, 5 ) &  [ 5, 10, 0, 8, 3, 12, 2, 7, 4, 1, 11, 6, 9 ] & ( 3, 5,-)  \\   
\hline
(3,2,1) & ( 3, 2, 6) &  [ 4, 7, 11, 6, 1, 8, 3, 10, 2, 5, 0, 9 ] & ( 8, 3, 4 ) & 3 \\   
& ( 3, 6, 1) &  [ 6, 2, 9, 1, 8, 4, 0, 3, 7, 10, 5 ] & ( 4, 5,-) \\   
& ( 6, 2, 1) &  [ 2, 5, 1, 8, 4, 9, 6, 3, 0, 7 ] & (2, 4,-) \\ 
\hline
(3,2,2) & ( 3, 2, 7 ) & [ 6, 11, 3, 8, 0, 4, 7, 2, 12, 9, 1, 10, 5 ] & (4, 5, 6) & 3\\   
& ( 3, 6, 2 ) &  [ 5, 1, 4, 8, 11, 3, 7, 0, 9, 2, 6, 10 ] & (8, 4,-)  \\   
& ( 6, 3, 2 ) &  [ 7, 0, 3, 6, 11, 8, 4, 1, 10, 2, 5, 9 ] & (2, 7,-) \\  
\hline
(3,2,3) & ( 3, 2, 8 ) &  [ 6, 11, 7, 2, 13, 4, 9, 0, 3, 8, 12, 1, 10, 5 ] & (4, 5, 6 ) & 3 \\   
& ( 3, 6, 3 ) &  [ 11, 1, 6, 10, 2, 5, 9, 0, 3, 7, 12, 8, 4 ] & ( 2, 4,-) \\  
& ( 6, 2, 3) &  [ 11, 8, 5, 1, 10, 7, 0, 3, 6, 2, 9, 4 ]  & (2, 4,-)  \\  
\hline
(3,2,4) & ( 3, 2, 9) &  [ 3, 8, 13, 10, 5, 0, 12, 7, 11, 1, 6, 2, 14, 9, 4 ] & ( 2, 3, 4 ) & 3 \\ 
& ( 3, 6, 4 ) &  [ 6, 11, 7, 2, 13, 9, 12, 1, 5, 10, 0, 4, 8, 3 ] & ( 2, 3,-) \\  
& ( 6, 2, 4 ) &  [ 0, 3, 7, 10, 1, 9, 4, 12, 2, 5, 8, 11, 6 ] & ( 4, 6,- ) \\  
\hline
(3,2,5) & ( 3, 2, 5) &  [ 3, 6, 1, 8, 0, 5, 10, 7, 2, 9, 4 ] & ( 2, 3, 4 ) & 1 \\   
\hline
(3,3,1) & ( 3, 3, 6 ) &  [ 3, 8, 11, 7, 12, 2, 6, 1, 9, 4, 0, 10, 5 ] & (  2, 3, 5 ) & 3 \\   
& ( 3, 7, 1 ) & [ 3, 7, 11, 6, 2, 10, 1, 4, 0, 8, 5, 9 ] & ( 2, 3,-) \\   
& ( 6, 3, 1 ) &  [ 9, 5, 8, 0, 4, 1, 7, 10, 2, 6, 3 ] & (  2, 3,-)  \\  
\hline
(3,3,2) & (3, 3, 7) &  [ 11, 2, 6, 1, 10, 13, 4, 9, 0, 3, 7, 12, 8, 5 ] & ( 9, 4, 5  ) & 3 \\
& ( 3, 7, 2 ) &  [ 6, 10, 1, 9, 5, 2, 11, 8, 4, 0, 3, 12, 7 ] & ( 6, 7,-) \\ 
& ( 6, 3, 2 ) &  [ 1, 10, 7, 11, 2, 6, 9, 4, 0, 3, 8, 5 ] & (4, 5,-)  \\  
\hline
(3,3,3) & ( 3, 3, 8 ) &  [ 9, 12, 2, 7, 10, 14, 4, 8, 13, 3, 6, 11, 1, 5, 0 ] & ( 8, 9, 4 ) & 2 \\  
& ( 3, 3, 3 ) &  [ 7, 3, 8, 2, 5, 0, 4, 1, 6, 9 ] & ( 2, 3,-)  \\ 
\hline
(3,3,4) & (  3, 3, 4 ) &  [ 5, 10, 7, 2, 9, 4, 0, 8, 1, 6, 3 ] & ( 2, 3, 5 ) & 1 \\  
\hline
(3,3,5) & ( 3, 3, 5 ) & [ 6, 11, 2, 7, 0, 9, 5, 1, 10, 3, 8, 4 ] & ( 3, 4, 6 ) & 1 \\  
\hline
(3,4,1) & ( 3, 4, 6) &  [ 2, 5, 0, 11, 6, 1, 10, 7, 3, 12, 8, 4, 13, 9 ] & ( 8, 9, 4 ) & 3 \\  
& ( 3, 8, 1) &  [ 7, 11, 3, 12, 2, 6, 9, 5, 1, 10, 0, 4, 8 ] & ( 6, 7,- )  \\ 
& ( 6, 4, 1) &  [ 4, 7, 3, 0, 8, 11, 2, 6, 9, 1, 10, 5 ] & ( 4, 5,-)  \\   
\hline
(3,4,2) & ( 3, 4, 7 ) &  [ 3, 8, 13, 10, 7, 12, 1, 5, 0, 11, 6, 2, 14, 9, 4 ] & (2, 3, 4) & 2 \\   
& ( 3, 4, 2 ) &  [ 5, 8, 4, 9, 2, 6, 1, 7, 3, 0 ] & (4, 5,-)  \\   
\hline
(3,4,3) & ( 3, 4, 3 ) &  [ 4, 8, 0, 7, 3, 9, 1, 6, 2, 10, 5 ] & (3, 4, 5) & 1 \\  
\hline
(3,4,4) & ( 3, 4, 4 ) &  [ 11, 7, 3, 10, 2, 5, 0, 9, 6, 1, 8, 4] & ( 8, 3,  4) & 1 \\   
\hline
(3,4,5) & ( 3, 4, 5) &  [ 4, 7, 11, 3, 8, 12, 2, 5, 0, 9, 1, 6, 10 ] & ( 8, 3, 4) & 1  \\
\hline
\end{array}$$
\end{footnotesize}
\end{center}
\end{table}

\begin{cor}\label{cor:max5}
Suppose $|U| = 3$ and $\max{(U)} \leq 5$.  Then $\BHR(L)$ holds for multisets with underlying set~$U$. 
\end{cor}

\begin{proof}
It is known that $\BHR(L)$ holds when the underlying set of`$L$ is a subset of $\{1,2,3,4\}$ or $\{1,2,3,5\}$~\cite{OPPS2, PP14b}.  Hence we may assume~$U = \{ x,4,5 \}$ with $x \leq 3$.  The case $\{ 1,4,5 \}$ is covered in~\cite{OPPS2} and $\{2,4,5 \}$ and $\{3,4,5\}$ are covered by Theorem~\ref{th:245_345}.
\end{proof}

\section{Underlying Set $\{ 1,2, x \}$ }\label{sec:12x}

In this section we study underlying sets of the form $\{ 1,2,x \}$ with~$x$ even.  However, before narrowing our focus, we give the more general result that is behind the work of this section.

\begin{lem}\label{lem:many_x} 
Let $U$ be a set with $\max{(U)} = \mu$.  Let $L$ be an admissible multiset of size~$v-1$ with underlying set~$U$ and suppose $x$ appears at least $(\mu -1)v / \mu$ times in~$L$.  Then a realization of~$L$ is necessarily $x$-growable.
\end{lem}

\begin{proof}
Suppose that there are $\mu$ consecutive elements $h_{i+1}, h_{i+2}, \ldots, h_{i + \mu}$ in a realization of~$L$ that are each incident with two edges of length~$x$.  Then the growth embedding at $m = i + \mu$ using~$x$ lengthens exactly one of the edges incident with each of the elements $h_{i + \mu -x + 1} , \ldots, h_{i + \mu}$.  Any other edge being lengthened must have length greater than~$\mu$, and so there are none and the realization is $x$-growable at~$i + \mu$.

Each element is incident with one or two edges.  Considering the elements adjacent to just one edge as being adjacent to that edge and an edge of length~0, there are $2v$ incidences in total.  At most $2v - (\mu -1)v / \mu = v - v/\mu$ of these are not with edges of length~$x$.  Hence there are insufficiently many non-$x$ incidences to avoid having a run of $\mu$ elements each incident with two edges of length~$x$.
\end{proof}

The following result allows Lemma~\ref{lem:many_x} to be useful for the underlying sets of interest in this section.

\begin{thm}\label{th:pp12x}{\rm \cite{PP14}}
Let $x \geq 4$ be even.  The multiset $\{ 1^a, 2^b, x^c \}$ satisfies the BHR Conjecture if $a+b \geq x-1$.
\end{thm}

We are therefore concerned here with multisets $\{ 1^a, 2^b, x^c \}$ that have even $x \geq 4$ and with $a+b < x-1$.  We start by putting an upper bound on the smallest counterexample.

\begin{thm}
Let $x \geq 4$ be even and $L = \{ 1^a, 2^b, x^c \}$.  If there is a counterexample to the BHR Conjecture for a multiset of this form, then there is one with $a+b < x-1$ and $c < x^2 - x + 1$.
\end{thm}

\begin{proof}
First, by Theorem~\ref{th:pp12x}, every counterexample must have $a+b < x-1$.

We apply Lemma~\ref{lem:many_x}: any realization with underlying set~$U$ and $c \geq  xv/(x-1)$ must be $x$-growable.  As $v = a+b+c+1$ and $a+b \leq x-2$, this simplifes to $c \geq x^2 -2x +1$.  Therefore, if the BHR Conjecture holds for all $c$ in the range $x^2 -2x +1  \leq c < x^2 -x +1$ it holds for all greater values of~$c$ too, by the $x$-growability of the realizations for these values.  Hence if there is a counterexample, then there must be one with $c < x^2 - x + 1$, as required.
\end{proof}

Therefore, to prove $\BHR(L)$ for $U = \{ 1,2,x \}$ with $x$ even, it suffices to show that there is no counterexample with $v \leq (x-2) +  (x^2 - x)  + 1 = x^2 -1$, a finite process.  In any given case, we expect to achieve the result by checking many fewer cases than potentially required by this upper bound.  Indeed, in performing the computations for the following result, we did not encounter a situation where a multiset is admissible but an $x$-growable realization does not exist.

\begin{thm}\label{th:12x}
Let $4 \leq x \leq 12$ with $x$~even.  $\BHR(L)$ holds for multisets~$L$ with underlying set $\{ 1,2,x \}$. 
\end{thm}

\begin{proof}
The result is proved for $x \in \{ 4,6,8 \}$ in~\cite{PP14}.  

For each of the remaining cases we do the following.  For each pair $a,b$ with $a+b < x-1$, let~$k$ be the smallest value such that $\{1^a, 2^b, x^k \}$ is admissible.  We find $x$-growable realizations for $\{1^a, 2^b, x^c \}$ for $k \leq c \leq k+x-1$, excluding values of~$c$ such that  $\{1^a, 2^b, x^{c+ix} \}$ is inadmissible for all $i \geq 0$.  This covers all admissible multisets.

The required realizations are given in a text file on the ArXiv page for this paper, or are available on request from the authors.  There are~249 for the $x=10$ case and~487 for the $x=12$ case.  
 \end{proof}


\section{An Algorithm}\label{sec:alg}

The third and final approach to underlying sets of size~3 has much in common with that of Section~\ref{sec:max5}.  As in that section, let $U = \{ x,y,z \}$ and $L = \{ x^a,y^b,z^c \}$.  We again divide the work for a given~$U$ into $\pi(U)$ cases according to the combinations of congruences of~$a$,~$b$ and~$c$ modulo~$x$,~$y$ and~$z$ respectively.

Suppose we are working on one case and have covered all $(a,b,c)$ with $a < a_0$, $b < b_0$ or $c < c_0$ for some $(a_0, b_0, c_0)$ in the appropriate equivalence class modulo~$(x,y,z)$.  As in Section~\ref{sec:max5}, define $a_i$ to be $a_0 + ix$ and make similar definitions for~$b_i$ and $c_i$ with respect to~$y$ and~$z$ respectively.

Consider the following steps: 
\begin{enumerate}
\item  Find an $\{x,y\}$-growable realization for $(a_I, b_J, c_0)$ for some $I,J \geq 0$.

\item For each~$i$ with $0 \leq i < I$, find a $y$-growable realization for $(a_i, b_{j_i}, c_0)$, unless $(a_i, b, c_0)$ is inadmissible for all~$b$, and a realization for $(a_i, b_j, c_0)$ for all admissible triples with $0 \leq j < {j_i}$.

\item  For each~$j$ with  $0 \leq j < J$, find an $x$-growable realization for $(a_{i_j}, b_{j}, c_0)$, unless $(a, b_j, c_0)$ is inadmissible for all~$a$, and a realization for $(a_i, b_j, c_0)$ for all admissible triples with $I \leq i < {i_j}$.
\end{enumerate}

If we complete these steps, we claim that we have now also covered all $(a,b,c)$ with $c = c_0$:  Consider admissible $(a,b,c)$ and set $c = c_0$.  If $a$ and $b$ are large (that is, $a \geq a_I$ and $b \geq b_J$) then Item~1 implies the existence of the required realization.  If $a$ is small (that is, $a < a_I$), then Item~2 implies the existence of the required realization.  Finally, if $b$ is small (that is, $b < b_J$) then we get the required realization from Item~3.

The algorithm works as follows.  Start by setting $(a', b', c')$ with $0 < a' \leq x$, $0 < b'  \leq y$ and $0 < c' \leq z$ (that is, set these to the original $(a_0, b_0, c_0)$ of Section~\ref{sec:max5}).  If $(a',b',c')$ has a $U$-growable realization we have completed this case.  If not, complete the above steps, and we have completed all cases with $c = c'$.  We may now begin again by looking for a $U$-growable realization for $(a', b', c' + z)$ and continuing into the steps above if unsuccessful.  

There is no guarantee that the algorithm terminates.  However, if it does, we have a finite list of realizations that proves that case of the conjecture for multisets with underlying set~$U$.   

There are places in the algorithm where there are choices to be made.  These choices can help the algorithm run more quickly (or terminate at all).  The most important of these is that at each run through the steps we may choose which elements of the underlying sets are playing the roles of $x$, $y$ and $z$.  We may always choose~$z$ to be (one of) the elements that occurs least frequently in the multiset to avoid the possibility of only looking at the family $(a', b', c + kz)$, which might never have a $U$-growable realization.  In fact, our implementation begins at the end by first finding a $U$-growable realization for some $(a,b,c)$ with~$v$ as small as possible.  It then makes choices of orderings of~$x$,~$y$ and~$z$ that head towards this known finishing point.

Before turning to the algorithm to provide results, we give a theoretical construction that will make the process smoother.

\begin{lem}\label{lem:12x}
Let $L = \{ 1^a, 2^b, x \}$.  If~$L$ is admissible then it is realizable.  
\end{lem}

\begin{proof}
By Corollary~\ref{cor:max5}, Theorem~\ref{th:12x} and the definition of edge-length we may assume that~$6 < x \leq v/2$.  We may also assume that $a,b \geq 1$, as the BHR~Conjecture is settled for all two-element sets~\cite{DJ09,HR09}, and that~$x$ is odd, as the BHR Conjecture is settled for $\{ 1^a, 2^b, x^c \}$ when~$x$ is even and $a+b \geq x-1$~\cite{PP14} 

We first define a sequence that has the multiset $\{1, 2^{k-1} \}$ as differences.  If~$k$ is odd, let
$$\sigma_k = 0, 2, 4, \ldots, k-1, k, k-2, k-4, \ldots, 1$$
and if $k$ is even let
$$\sigma_k = 0, 2, 4, \ldots,  k, k-1, k-3, k-5, \ldots, 1.$$
Denote the sequence obtained  from~$\sigma_k$ by adding~$y$ to each element to be $\sigma_k + y$.

Given a sequence that ends with~$k+1$ consecutive numbers $y, y+1, \ldots, y + k$, we may replace those numbers with $\sigma_k + y$ to replace the~$k-1$ of realized differences of~1 with $k-1$ realized differences of~2.

Consider the sequence
$$ v-1, v-2, \ldots, x ; 0, 1, \ldots x-1. $$
This is a Hamiltonian path that realizes~$\{1^{v-2}, x\}$ and either side of the semi-colon is a sequence of consecutive numbers.  We may replace subsequences at one or both ends with sequences of the form $\sigma_k+y$ to obtain a sequence that realizes $\{ 1^a, 2^b, x \}$ for any~$a$ in the range~$2 \leq a \leq v-2$.

This leaves the case~$a=1$.  If~$v$ is even then the Hamiltonian path
$$ v-2, v-4, \ldots, x+1, x-1, x-2, x-4, \ldots, 1, v-1, v-3,  \ldots, x, 0, 2, 4, \ldots, x-3$$
where all differences are~2 except between $x-1$ and $x-2$ and between $x$ and 0, realizes~$\{1, 2^{v-3}, x \}$.  If $v$ is odd, the same is true of
$$ v-2, v-4, \ldots, x+2, x, 0, 2, 4, \ldots, x-1, x-2, x-4, \ldots, 1, v-1, v-3, \ldots, x+1.$$
\end{proof}

\begin{exa}\label{ex:12x}
Let $v=16$ and $x=7$; here are some realizations of multisets of the form~$\{ 1^a, 2^b, 7 \}$ obtainable from the proof of~Lemma~\ref{lem:12x}.  The initial Hamiltonian path that realizes $\{ 1^{14}, 7 \}$ is
$$[ 15, 14, 13, 12 ,11, 10, 9, 8, 7, 0, 1,2,3,4,5,6 ].$$
To realize $\{ 1^{10}, 2^4, 7 \}$, we can replace the last six elements with $\sigma_5 + 1$ as follows:
$$[ 15, 14, 13, 12 ,11, 10, 9, 8, 7, 0, 1, 3, 5, 6, 4, 2 ].$$
To realize $\{ 1^{7}, 2^7, 7 \}$ we can replace the first five elements of this new sequence with $\sigma_4 + 11$, reversed:
$$[ 12, 14, 15, 13 ,11, 10, 9, 8, 7, 0, 1, 3, 5, 6, 4, 2 ].$$
The Hamiltonian path that realizes~$\{1, 2^{12}, 7 \}$ is
$$[ 14, 12, 10, 8,6, 5,3,1, 15,13, 11, 9, 7, 0,2,4 ].$$
\end{exa}

We can now use the algorithm to prove the main result of the section.

\begin{thm}\label{th:alg}
$\BHR(L)$ holds for multisets~$L$ whose underlying set is one of the following~$25$ sets of size~$3$:
$$ \{1,2,7\}, \{1,2,9\},  \{1,3,6\} ,  \{1,3,7\},   \{1,3,8\}, \{1,4,6\},   \{1,4,7 \},  \{ 1,5,6\},  $$
$$ \{ 1,5,7 \} , \{ 1,6,7 \} , \{2,3,6\}  , \{2,3,7 \}, \{2,4,7\},  \{2,5,6\},  \{2,5,7\}, \{2,6,7 \},  $$
$$  \{3,4,6\},  \{3,4,7\},  \{3,5,6\}, \{3,5,7\}, \{3,6,7 \},  \{4,5,6\} , \{4,5,7\}, \{4,6,7\}, \{5,6,7\}  . $$
\end{thm}

\begin{proof}

We use the above algorithm, implemented in GAP and available on the ArXiv page for this paper or by request from the authors, to generate most of the required realizations to prove the result.   

The algorithm is unsuccessful as it stands for underlying sets of the form $\{ x,3,6 \}$.   In particular, $\{3,6\}$-growable realizations for multisets of the form~$\{3^a, 6^b, x \}$ appear to not exist.  We circumvent this issue by using Lemma~\ref{lem:12x} as follows.

A multiset $L = \{3^a, 6^b, x \}$ is admissible only if $v = a+b+1+1$ is not a multiple of~3.  In this case, division by~$3 \pmod{v}$ is an automorphism of~$\Z_v$ and applying it to~$L$ we get an equivalent problem for the multiset $L' = \{ 1^a, 2^b, \pm x/3 \}$.  This equivalent problem has a realization by Lemma~\ref{lem:12x}.

Hence when applying the algorithm to underlying sets of the form~$\{ 3,6,x \}$ we may assume that~$x$ appears at least twice in any multiset we need to investigate.  This is sufficient for the algorithm to go through for the five cases of this form here.

In the order of the statement of the theorem, the number of realizations produced by the program for each underlying set is
$$ 198, 578, 140, 263, 402, 226, 361 , 296,     $$
$$  405, 542, 172, 310, 453, 314, 718, 576,     $$
$$  285, 567,  367, 598, 676, 390, 685, 808,  774  $$
for a total of 11,104 realizations.
The file containing all of the realizations required for this proof in GAP-readable format is also available on the ArXiv page or by request from the authors.
\end{proof}

\section{Concluding Remarks}\label{sec:final}

We now have everything we need to prove the main result.

\begin{proof}[Proof of Theorem~\ref{th:main}]
Let $U = \{ x,y,z \}$ and $L = \{ x^a, y^b, z^c \}$.   If $\max{(U)} \leq 5$ then Corollary~\ref{cor:max5} gives the result.  

Consider $U$ with $\max{(U)} = 6$.  The set $\{1,2,6\}$ is covered in~\cite{PP14}.  If $U = \{2,4,6\}$ and $v$ is even, then~$L$ is inadmissible.  If $v$ is odd then we may take the solution for $\{ 1^a, 2^b, 3^c \}$ from~\cite{CD10} and apply the $\Z_v$-automorphism $t \mapsto 2t$.  Otherwise, $U$ is covered by Theorem~\ref{th:alg}.

All cases with $\max{(U)} = 7$ are directly covered by Theorem~\ref{th:alg}.   When $\max{(U)} > 7$ we have five cases to consider: $\{ 1,2,8 \}$, $\{1,2,9\}$, $\{1,2,10\}$, $\{1,2,12\}$ and $\{1,3,8\}$.  The set $\{1,2,8\}$ is covered in~\cite{PP14}, the sets~$\{ 1,2, 9\}$ and~$\{1,3,8\}$ are covered by Theorem~\ref{th:alg}, and the sets~$\{1,2,10\}$ and~$\{1,2,12\}$ are covered by Theorem~\ref{th:12x}.
\end{proof}

We have presented three approaches to completely solving instances of the BHR~Conjecture for a given three-element underlying set.   Each has its advantages and its drawbacks when considering how they might make further inroads into the problem.

The approach of Section~\ref{sec:max5} has the advantage of being human-readable and allows insight into the structure of the problem at hand to reduce the number of realizations required for a particular underlying set.  However, the number of realizations that are required is large and will quickly grow larger as the elements in the underlying sets do.  More insight into how to construct growable realizations is needed to get anywhere near being able to implement the approach without resorting to a computer search.

We believe that the approach of Section~\ref{sec:12x} offers the most hope for the most substantial theoretical advances: coupling non-growable techniques with growable ones gives a way to harness other theoretical approaches and complement them with the theory of growable realizations.  Of course, this requires the concomitant development of other theoretical approaches.

The approach of Section~\ref{sec:alg} also offers hope.  Might one prove that the steps of the algorithm are guaranteed to succeed, hence proving the conjecture for some underlying sets (perhaps only for a portion of the $\pi(U)$ cases, or with a finite number of missing subcases) without needing to run the programs?  If there is a specific three-element underlying set of interest, with elements that are not too large, then the Section~\ref{sec:alg} approach is probably the one that will most quickly lead to a resolution, given current tools.

\section*{Acknowledgements}

We are grateful to Onur Agirseven, Kat Cannon-MacMartin, Eamon Mahoney, Anita Pasotti, Marco Pellegrini and John Schmitt for conversations about the BHR Conjecture.   The majority of this work was completed in the summer of~2021 under the  auspices of the School of Communication Research Co-Curricular Program at Emerson College.

\end{document}